\documentclass[a4paper,11pt]{amsart}
\usepackage{etex}

\usepackage{a4wide}
\usepackage{enumerate}
\usepackage{url}
\usepackage[all]{xy}

\usepackage{latexsym, amssymb, amsthm,pictex}
\usepackage[centertags]{amsmath}
\usepackage[dvips]{color}
\usepackage{graphicx}
\usepackage[small]{caption}

\newcommand{\ee}{\end{equation}}
\newcommand{\eea}{\end{eqnarray}}
\newcommand{\bean}{\begin{eqnarray*}}
\newcommand{\eean}{\end{eqnarray*}}

 \newif\ifpctex

 \newcommand{\eps}{\epsilon}

\newcommand{\smallx}{\mathpzc{x}}

\DeclareMathAlphabet{\mathpzc}{OT1}{pzc}{m}{it}

  \newtheorem{lemma}{Lemma}[section]
  \newtheorem{theorem}{Theorem}
  \newtheorem{proposition}[lemma]{Proposition}
  \newtheorem{cor}[lemma]{Corollary}
  \newtheorem{corollary}[lemma]{Corollary}

  \newtheorem{cond}[lemma]{Condition}
    
  \theoremstyle{definition}
  \newtheorem{definition}[lemma]{Definition}
  \newcommand{\beCond}[2]{\Rand{\vspace{0,6cm}\tt #1}\begin{cond}[#2]
  \label{#1}} \theoremstyle{definition}
  \newtheorem{remarkx}[lemma]{Remark}
  \newtheorem{examplex}[lemma]{Example}
  \numberwithin{equation}{section}
  \newtheoremstyle{step}{3pt}{0pt}{\itshape}{}{\bf}{}{.5em}{}

\theoremstyle{step} \newtheorem{step}{Step}

\newcommand{\CD}{\mathcal{D}} 
 
\newcommand{\R}{\mathbb{R}} 
\newcommand{\N}{\mathbb{N}} \newcommand{\M}{\mathbb{M}}

\newcommand{\Rand}[1]{\marginpar{#1}} \marginparwidth1.5cm
\newcommand{\be}[1]{\begin{equation}\label{#1}}
\newcommand{\bew}[1]{\Rand{\vspace{0,6cm}\tt
#1}\begin{equation*}\label{#1}}
\newcommand{\bea}[1]{\Rand{\vspace{0,6cm}\tt
#1}\begin{eqnarray}\label{#1}}

\newcommand{\beL}[2]{\Rand{\vspace{0,6cm}\tt
#1}\begin{lemma}[#2]\label{#1}}
\newcommand{\beD}[2]{\Rand{\vspace{0,6cm}\tt
#1}\begin{definition}[#2]\label{#1}}
\newcommand{\beT}[2]{\Rand{\vspace{0,6cm}\tt
#1}\begin{theorem}[#2]\label{#1}}
\newcommand{\beP}[2]{\Rand{\vspace{0,6cm}\tt
#1}\begin{proposition}[#2]\label{#1}}
\newcommand{\beC}[2]{\Rand{\vspace{0,6cm}\tt #1}\begin{cor}[#2]\label{#1}}

\newcommand{\Tto}{{_{\displaystyle\Longrightarrow\atop t\to\infty}}}

\newcommand{\Tno}{{_{\displaystyle\Longrightarrow\atop n\to\infty}}}

\newcommand{\tno}{{_{\displaystyle\longrightarrow\atop n\to\infty}}}

\newcommand{\tNo}{{_{\displaystyle\longrightarrow\atop N\to\infty}}}

 \usepackage{mathtools}
\usepackage{pstricks}
\usepackage{multido}
\usepackage{dsfont}
\usepackage{color}              

\definecolor{usablegreen}{rgb}{0,.5,0}
\definecolor{usablecyan}{rgb}{0,.5,.5}

\newcommand{\newatop}[2]{\underset{\scriptscriptstyle #2}{#1}}	\newcommand{\convto}[2][\infty]{\,\newatop{\longrightarrow}{#2 \to #1}\,}	\newcommand{\Convto}[2][\infty]{\,\newatop{\Longrightarrow}{#2 \to #1}\,}
\renewcommand{\tno}{\convto{n}}
\renewcommand{\Tno}{\Convto{n}}
\renewcommand{\tNo}{\convto{N}}
\renewcommand{\Tto}{\Convto{t}}

\newcommand{\comment}[1]{}

\newcommand{\pstep}[1]{\removelastskip\smallskip\par\noindent\emph{#1.} \hspace{0.25ex}}

\newcommand{\nbd}{\protect\nobreakdash-\hspace{0pt}}
\newcommand{\nclad}[1][m]{$#1$\nbd labelled cladogram}
\newcommand{\nclads}[1][m]{\nclad[#1]s}

\DeclareMathOperator{\Dir}{Dir}
\newcommand{\Dirhalf}{\Dir(\tfrac12,\ldots,\tfrac12)}

\newcommand{\Exp}{\mathbb{E}}
\newcommand{\one}{\mathds{1}}

\newcommand{\testtree}{\mathfrak{t}}

\newcommand{\Etk}{E_{\testtree,k}(\uu)}

\newcommand{\tk}{\testtree_{\land k}}
\newcommand{\uk}{{\uu_{\land k}}}
\newcommand{\uu}{\underline{u}}
\newcommand{\ux}{\underline{x}}

\newcommand{\X}{\mathpzc{X}}

\newcommand{\Clad}[1][m]{\mathfrak{C}_{#1}}
\newcommand{\shape}[1][T]{\mathfrak{s}_{#1}}

\newcommand{\sPol}{\Pi_{\shape[]}}

\renewcommand{\(}{\bigl(} 	\renewcommand{\)}{\bigr)}

\newcommand{\T}{\mathbb{T}}
\newcommand{\Tbin}{\T_{2}}

\newcommand{\mut}{\tilde\mu}

\newcommand{\Tb}{\overline{T}}
\newcommand{\cb}{\bar{c}}

\DeclareMathOperator{\lf}{lf}

\DeclareMathOperator{\br}{br}

\DeclareMathOperator{\edge}{edge}
\DeclareMathOperator{\intedge}{in-edge}
\DeclareMathOperator{\extedge}{ex-edge}

\DeclareMathOperator{\at}{at}

\newcommand{\restricted}[1]{{\mspace{-1mu}\upharpoonright}_{#1}}
\newcommand{\set}[2]{\{#1: #2\}}
\newcommand{\bset}[2]{\bigl\{#1 \,:\, #2\bigr\}}
\newcommand{\Bset}[2]{\Bigl\{#1 \>:\> #2\Bigr\}}
\newcommand{\openint}[2]{\mathopen]#1, #2\mathclose[}

\newcommand{\Rtree}{$\R$\nobreakdash-tree}

\newcommand{\Sub}{\mathcal{S}}
\newcommand{\dx}{\mathrm{d}}

\newcommand{\C}{\mathcal{C}}

\newenvironment{remark}
  {\pushQED{\qed}\remarkx}
  {\popQED\endremarkx}
\newenvironment{example}
  {\pushQED{\qed}\examplex}
  {\popQED\endexamplex}

\newenvironment{enremark}[1][]{\begin{remark}[#1]\begin{enumerate}[1.]}{\qedhere\end{enumerate}\end{remark}}

\newcommand{\node}{\bullet}
\newcommand{\xyedge}{\ar@{-}}
\newcommand{\leaf}{{}\phantom{\bullet}}
\newcommand{\Bnode}{\circ}
\newcommand{\Rnode}{\times}
\newcommand{\lfdr}{\leaf\ar@{-}[dr]}
\newcommand{\lfd}{\leaf\ar@{-}[d]}
\newcommand{\lfur}{\leaf\ar@{-}[ur]}
\newcommand{\lful}{\leaf\ar@{-}[ul]}
\newcommand{\lfdl}{\leaf\ar@{-}[dl]}
\newcommand{\brd}{\Bnode\ar@{-}[d]}
\newcommand{\brrr}{\node\ar@{-}[rrr]}
\newcommand{\brr}{\Bnode\ar@{-}[r]}
\newcommand{\brl}{\Rnode\ar@{-}[l]}
\newcommand{\brdr}{\Bnode\ar@{-}[dr]}
\newcommand{\brdl}{\node\ar@{-}[dl]}
\newcommand{\brul}{\Rnode\ar@{-}[ul]}
\newcommand{\Rbrdl}{\Rnode\ar@{-}[dl]}

\newcommand{\colsepdefault}{0.75pc}
\xymatrixrowsep{0.75pc} \xymatrixcolsep{\colsepdefault}

\newcommand{\cfigure}[3]{
	\begin{figure}
	\centerline{
	 #2 
	 }
	\caption{#3}
	\label{#1}
	\end{figure}
}
\newcommand{\xymatfig}[3]{\cfigure{#1}{$\xymatrix{#2}$}{#3}}
\newcommand{\xymatdoublefig}[5][\colsepdefault]{\cfigure{#2}{\xymatrixcolsep{#1}
	(a) $\,\xymatrix{#3}$ \hfil (b) $\,\xymatrix{#4}$}{#5}}

\makeatletter 
\renewcommand\p@enumii{}
\makeatother

\renewcommand{\shape}[1][(T,c)]{\mathfrak{s}_{#1}}
\newcommand{\AldGen}{\Omega_\mathrm{Ald}}	\newcommand{\AldGenMass}{\Omega_\mathrm{Ald}}							
\newcommand{\Taldous}{\mathbb{T}_2^{\mathrm{cont}}}	\newcommand{\mlClad}[1][m]{\overline{\mathfrak{C}}_{#1}}	\DeclareMathOperator{\linspan}{span}

\newcommand{\ueta}{\underline\eta}
\newcommand{\uetaz}{\ueta^z}
\newcommand{\tke}{\testtree^{(k,e)}}
\title[The Aldous chain on cladograms in the diffusion limit]{Aldous chain on cladograms in the diffusion limit}

\date{\today}



\author{Wolfgang L\"ohr}
\address{Wolfgang L\"ohr\\
Universit\"at Duisburg-Essen\\ Fakult\"at f\"ur Mathematik \\ 45117 Essen \\ Germany \\
and
 TU Chemnitz\\ Fakult\"at f\"ur Mathematik \\ 09107 Chemnitz \\ Germany}
\email{wolfgang.loehr@uni-due.de}

\author{Leonid Mytnik}
\address{Leonid Mytnik\\
William Davidson Faculty of Industrial Engineering and Management\\
Israel Institute of Technology Technion\\
Haifa 32000\\ Israel}
\email{leonid@ie.technion.ac.il}

\author{Anita Winter}
\address{Anita Winter\\
Universit\"at Duisburg-Essen\\ Fakult\"at f\"ur Mathematik \\ 45117 Essen \\ Germany}
\email{anita.winter@uni-due.de}

\thanks{Research supported by DFG RTG 2131, DFG SPP 1590, Israel Science Foundation (grants No.\ 1325/14,
1704/18), Aly Kaufman fellowship for Anita Winter and visiting fellowship of the Technion for Wolfgang L\"ohr.
Wolfgang L\"ohr was partially supported by the DFG project 415705084.}

\begin{document}

\begin{abstract}  In \cite{MR1774749}, Aldous investigates a symmetric Markov chain
  on cladograms and gives bounds on its mixing and relaxation times. The latter bound was sharpened in \cite{MR1871950}.
  In the present paper we encode cladograms as binary, algebraic measure trees and show that this Markov chain
  on cladograms with a fixed number of leaves converges in distribution as the number of leaves tends to infinity.
  We give a rigorous construction of the limit as the solution of a well-posed
  martingale problem. The existence of a continuum limit diffusion was
  conjectured by Aldous, and
  we therefore refer to it as Aldous diffusion.  We show that the Aldous diffusion is a Feller process with continuous paths, and the
  algebraic measure Brownian CRT is its unique invariant distribution.

  Furthermore, we consider the vector of the masses of the three subtrees connected to a sampled branch point.
  In the Brownian CRT, its annealed law is known to be the Dirichlet distribution.
  Here, we give an explicit expression for the infinitesimal evolution of its quenched law under the Aldous
  diffusion.
\end{abstract}

\maketitle

\vspace{-\baselineskip}\begin{quote} {\footnotesize \tableofcontents} \end{quote}

 \section{Introduction}
\label{S:intro}
An \emph{$N$-cladogram} is a semi-labelled, unrooted, binary tree with $N\ge 2$ \emph{leaves} labelled
$\{1,2,...,N\}$ and with $N-2$ unlabelled internal vertices. 
Cladograms are particular phylogenetic trees for which no information on the edge lengths is available, and
which therefore only capture the tree structure. Reconstructing cladograms from DNA data is of major interest in
population genetics. An important ingredient for several algorithms are Markov chains that move through a space
of finite trees (see, for example, \cite{Felsenstein2003} for a survey on Markov chain Monte Carlo algorithms in
maximum likelihood tree reconstruction).
Usually, such chains are based on a set of simple rearrangements that transform a tree into a ``neighboring''
tree (see, for example, \cite{Felsenstein2003, MR1920237, MR1867931, MR1841949}).

The present paper considers (a continuous-time version of) the \emph{Aldous chain} on cladograms, which
is a Markov chain on the space $\Clad[N]$ of all $N$-cladograms. It has the following transition rates: for each
pair $(u,e)$ consisting of a leaf and an edge, at rate $1$, the
Markov chain jumps from its current state $\mathfrak{t}$ to $\mathfrak{t}^{(u,e)}$, where that latter is
obtained as follows (see Figures~\ref{Fig:01} and \ref{Fig:02}). If $u$ is not incident to $e$, then
\begin{itemize}
\item
Erase the unique edge (including the incident vertices) incident to $u$,
\item
split the remaining subtree at the edge $e$ into two pieces, and
\item
reintroduce the above edge (including $u$ and the branch point) at the split point.
\end{itemize}
Otherwise, if $u$ is incident to $e$, we set $\testtree^{(u,e)}=\testtree$. In total, these so-called
\emph{Aldous moves} from $\testtree$ to $\testtree^{(u,e)}$ happen at rate $N(2N-3)$, and the rate of actual
jumps of the Markov chain (where $\testtree^{(u,e)}\ne\testtree$) is $N(2N-6)$.

\begin{figure}
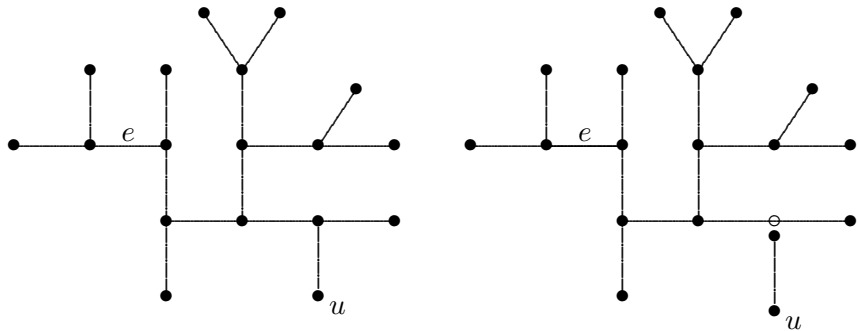

\centerline{
\beginpicture
	\setcoordinatesystem units <.5cm,.5cm>
	\setplotarea x from -0.5 to 22, y from -3.5 to 6
	\plot 4 0 10 0 /
	\plot 8 0 8 -1.78 /
	\plot 4 0 4 4 /
	\plot 2 2 0 2 /
	\plot 2 2 4 2 /
	\plot 2 2 2 4 /
	\plot 6 0 6 4 /
	\plot 6 2 10 2 /
	\plot 6 4 7 5.5 /
	\plot 8 2 9 3.5 /
	\plot 6 4 5 5.5 /
	\plot 4 0 4 -2 /
	\put{$\bullet$} [lC] at 3.8 -2
	\put{$\bullet$} [lC] at 4.8 5.5
	\put{$\bullet$} [lC] at 7.8 2
	\put{$\bullet$} [lC] at 9.8 2
	\put{$\bullet$} [lC] at 8.8 3.5
	\put{$\bullet$} [lC] at 6.8 5.5
	\put{$\bullet$} [lC] at 5.8 2
	\put{$\bullet$} [lC] at 5.8 4
			\put{$\bullet$} [lC] at 1.8 4
	\put{$\bullet$} [lC] at 1.8 2
	\put{$\bullet$} [lC] at -.2 2
	\put{$\bullet$} [lC] at 3.8 2
	\put{$\bullet$} [lC] at 3.8 4
	\put{$\bullet$} [lC] at 3.8 0
	\put{$\bullet$} [lC] at 5.8 0
	\put{$\bullet$} [lC] at 7.8 0
	\put{$\bullet$} [lC] at 7.8 -2
	\put{$u$} [IC] at 8.5 -2.3
	\put{$\bullet$} [lC] at 9.8 0
	\put{$e$} [IC] at 3 2.3
		\plot 16 0 22 0 /
	\plot 20 -.4 20 -2.18 /
	\plot 16 0 16 4 /
	\plot 16 2 14 2 /
	\plot 12 2 16 2 /
	\plot 14 2 14 4 /
	\plot 18 0 18 4 /
	\plot 18 2 22 2 /
	\plot 18 4 19 5.5 /
	\plot 20 2 21 3.5 /
	\plot 18 4 17 5.5 /
	\plot 16 0 16 -2 /
	\put{$\bullet$} [lC] at 15.8 -2
	\put{$\bullet$} [lC] at 16.8 5.5
	\put{$\bullet$} [lC] at 19.8 2
	\put{$\bullet$} [lC] at 21.8 2
	\put{$\bullet$} [lC] at 20.8 3.5
	\put{$\bullet$} [lC] at 18.8 5.5
	\put{$\bullet$} [lC] at 17.8 2
	\put{$\bullet$} [lC] at 17.8 4
			\put{$\bullet$} [lC] at 13.8 4
	\put{$\bullet$} [lC] at 13.8 2
	\put{$\bullet$} [lC] at 11.8 2
	\put{$\bullet$} [lC] at 15.8 2
	\put{$\bullet$} [lC] at 15.8 4
	\put{$\bullet$} [lC] at 15.8 0
	\put{$\bullet$} [lC] at 17.8 0
	\put{$\circ$} [lC] at 19.8 0
	\put{$\bullet$} [lC] at 19.8 -0.4
	\put{$\bullet$} [lC] at 19.8 -2.4
	\put{$u$} [IC] at 20.5 -2.7
	\put{$\bullet$} [lC] at 21.8 0
	\put{$e$} [IC] at 15 2.3
\endpicture
}
\caption{At rate $N(2N-3)$, a) a leaf $u$ and an edge $e$ are picked at random, and if $e$ and $u$ are not adjacent,
b) the edge incident to $u$ is taken away, leaving behind a branch point of degree $2$.
(continued in Figure~\ref{Fig:02})}
\label{Fig:01}
\end{figure}

This Markov chain has the generator $\Omega_N$, acting on all functions $\phi\colon \Clad[N]\to\R$ as follows:
\begin{equation}
\label{e:genN}
   \Omega_N\phi(\mathfrak{t})
   =
   \sum_{(u,e)}\Bigl(\phi\(\mathfrak{t}^{(u,e)}\)-\phi(\mathfrak{t})\Bigr),
\end{equation}
where the sum runs over all pairs $(u,e)$ consisting of a leaf and an edge, and $\testtree\in \Clad[N]$.
Obviously, the Aldous chain is reversible, and the uniform distribution on $\Clad[N]$ is the stationary distribution.
It was shown in \cite{MR1774749} that both mixing and relaxation time of the discrete-time chain are of order at
least ${\mathcal O}(N^2)$, but at most of order ${\mathcal O}(N^3)$. \cite{MR1871950} verified that the
relaxation time is of order ${\mathcal O}(N^2)$. Therefore, our continuous-time version has relaxation time of
order 1.

As \cite{Aldous1993} shows that a random $N$-cladogram with uniform edge lengths $\tfrac{1}{\sqrt{N}}$ converges
weakly to the Brownian Continuum Random Tree (CRT), Aldous conjectured the existence of a CRT-symmetric
diffusion limit of the Aldous chain on $N$-cladograms observed at time scale of order ${\mathcal O}(N^2)$ as
$N\to\infty$. This conjecture was presented in a talk in March 1999 given at the Fields Institute, and is supported by the following calculation: suppose we start the Markov chain in some
initial $N$-cladogram, fix a branch point, and consider the relative sizes $(\eta_1,\eta_2,\eta_3)$ of the
three subtrees attached to this branch point. Then, as the Markov chain runs, these proportions
change as a certain Markov chain, until the branch point disappears. On the proposed time-rescaling of $N^2$,
the $N \to \infty$ limit is the diffusion with generator
\begin{equation}
\label{e:FVnegative}
   \Omega f(\underline{\eta}) = \sum_{1\le i,j\le 3}\eta_i(\delta_{i,j}-\eta_i)\partial^2_{i,j}f(\underline{\eta})
   					-\tfrac12 \sum_{i=1}^3\(1-3\eta_i\)\partial_if(\underline{\eta})
\end{equation}
which records certain aspects of a diffusion on the continuum tree. Aldous raised the question of
how this diffusion should be constructed rigourously and what more can we calculate from there? On Aldous open problem website
the construction was rated as straightforward provided the right set-up is chosen.

\begin{figure}
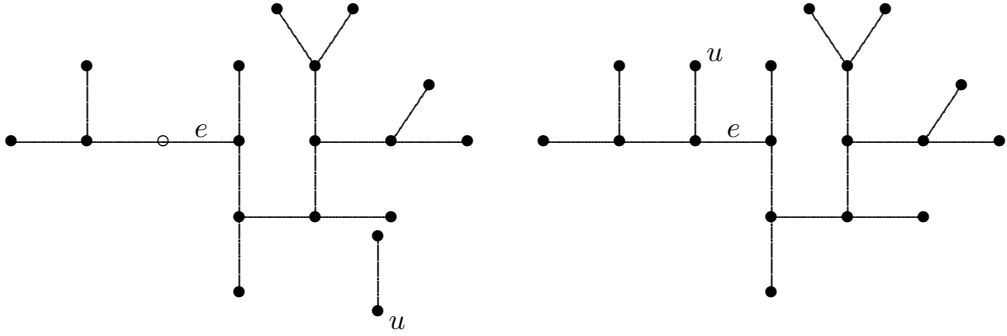

\centerline{
\beginpicture
	\setcoordinatesystem units <.5cm,.5cm>
	\setplotarea x from -2.5 to 12, y from -3.5 to 6
	\plot 4 0 8 0 /
	\plot 7.65 -2.5 7.65 -0.5 /
	\plot 4 0 4 4 /
	\plot 2 2 -2 2 /
	\plot 2 2 4 2 /
	\plot 0 2 0 4 /
	\plot 6 0 6 4 /
	\plot 6 2 10 2 /
	\plot 6 4 7 5.5 /
	\plot 8 2 9 3.5 /
	\plot 6 4 5 5.5 /
	\plot 4 0 4 -2 /
	\put{$\bullet$} [lC] at 3.8 -2
	\put{$\bullet$} [lC] at 4.8 5.5
	\put{$\bullet$} [lC] at 7.8 2
	\put{$\bullet$} [lC] at 9.8 2
	\put{$\bullet$} [lC] at 8.8 3.5
	\put{$\bullet$} [lC] at 6.8 5.5
	\put{$\bullet$} [lC] at 5.8 2
	\put{$\bullet$} [lC] at 5.8 4
		\put{$\bullet$} [lC] at -.2 4
	\put{$\circ$} [lC] at 1.8 2
	\put{$\bullet$} [lC] at -.2 2
	\put{$\bullet$} [lC] at -2.2 2
	\put{$\bullet$} [lC] at 3.8 2
	\put{$\bullet$} [lC] at 3.8 4
	\put{$\bullet$} [lC] at 3.8 0
	\put{$\bullet$} [lC] at 5.8 0
	\put{$\bullet$} [lC] at 7.8 0
	\put{$\bullet$} [lC] at 7.45 -0.5
	\put{$\bullet$} [lC] at 7.45 -2.5
	\put{$e$} [IC] at 3 2.3
	\put{{$u$}} [IC] at 8.15 -2.8
		\plot 18 0 22 0 /
	\plot 18 0 18 4 /
	\plot 16 2 12 2 /
	\plot 16 2 18 2 /
	\plot 14 2 14 4 /
	\plot 20 0 20 4 /
	\plot 20 2 24 2 /
	\plot 20 4 21 5.5 /
	\plot 22 2 23 3.5 /
	\plot 20 4 19 5.5 /
	\plot 18 0 18 -2 /
	\plot 16 2 16 4 /
	\put{$\bullet$} [lC] at 15.8 4
	\put{$u$} [IC] at 16.5 4.3
	\put{$\bullet$} [lC] at 17.8 -2
	\put{$\bullet$} [lC] at 18.8 5.5
	\put{$\bullet$} [lC] at 21.8 2
	\put{$\bullet$} [lC] at 23.8 2
	\put{$\bullet$} [lC] at 22.8 3.5
	\put{$\bullet$} [lC] at 20.8 5.5
	\put{$\bullet$} [lC] at 19.8 2
	\put{$\bullet$} [lC] at 19.8 4
		\put{$\bullet$} [lC] at 13.8 4
	\put{$\bullet$} [lC] at 15.8 2
	\put{$\bullet$} [lC] at 13.8 2
	\put{$\bullet$} [lC] at 11.8 2
	\put{$\bullet$} [lC] at 17.8 2
	\put{$\bullet$} [lC] at 17.8 4
	\put{$\bullet$} [lC] at 17.8 0
	\put{$\bullet$} [lC] at 19.8 0
	\put{$\bullet$} [lC] at 21.8 0
	\put{$e$} [IC] at 17 2.3
\endpicture
}
\caption{c) The two edges containing the branch point of degree $2$ are identified, the edge $e$ is
opened, and d) the free edge is reattached there.}
\label{Fig:02}
\end{figure}

\smallskip
The present paper is demonstrating that indeed a straightforward construction can be given once we choose the right
state space. A classical starting point would be to think of continuum trees as real trees which are particular
metric spaces. A metric space is called a real tree if it is path connected and satisfies the so-called
$4$-point condition. For convergence results one would like to be in a position to treat the approximating
discrete trees and their path-connected scaling limits in a unified way. One therefore also considers
\emph{metric trees} (introduced in \cite{AthreyaLohrWinter2017}), which are metric spaces differing from a real
tree by not necessarily being path connected. A metric space is a metric tree if it can be embedded
isometrically into a real tree in such a way that for every choice of three points in the metric tree, the
corresponding branch point (defined in the real tree) belongs to the metric tree.

In many applications it is useful to have metric trees equipped with a probability measure as, for example, the definition of the discrete Aldous chain dynamics requires to sample leaves according to some probability measure.
One therefore considers the space $\mathbb{M}$ of isometry classes of \emph{metric measure spaces} and equips it with
the Gromov-weak topology. In fact, Aldous's CRT arises as the Gromov-weak scaling limit of uniformly chosen
$N$-cladograms with the uniform distribution on the leaves and edge lengths scaled down by the factor $\tfrac{1}{\sqrt{N}}$.

One of the equivalent definitions of the \emph{Gromov-weak topology} is by convergence of the distance
matrix distributions, i.e., a sequence $(\smallx_N)_{N\in\N}$ of metric measure spaces converges to a metric
measure space $\smallx\in\M$ if and only if $\Phi(\smallx_N)\tNo\Phi(\smallx)$ for all test
functions $\Phi\colon \M \to \R$ of the form
\begin{equation}
\label{e:polyn}
   \Phi(\smallx):=\int\phi\((r(u_i,u_j))_{1\le i<j\le m}\)\,\mu^{\otimes m}(\mathrm{d}\uu),
\end{equation}
where $\smallx=(X,r,\mu)$, $m\in\N$, and $\phi\in{\mathcal C}_b(\R_+^{\binom m2})$ (see \cite{GrevenPfaffelhuberWinter2009,Loehr2013}).

In this set-up, many tree-valued Markov processes have been constructed and in some cases also the convergence
of approximating discrete tree-valued dynamics has been established (see, for example,
\cite{EvansWinter2006,GrevenPfaffelhuberWinter2013,DepperschmidtGrevenPfaffelhuber2012,KliemLoehr2015,LoehrVoisinWinter2015}). One could think that metric (measure) trees are
the natural framework for rescaling the Aldous chain as well. However, the Aldous chain resists this
approach. An easy calculation shows that the quadratic variation of the averaged distance process rescales at
time scale $N^{\frac32}$. But how does it relate to the conjecture that the Aldous chain rescales on the time
scale $N^2$? One reason might be that distances behave too wildly for tightness on that time scale to hold. Which in turn might be a hint that the naively used graph distance is not
the notion of distance intrinsic to the Aldous chain dynamics. And indeed, one can argue that two points are
close if the mass branching off the line segment connecting them is small rather than if the length of that line segment is small. The idea for our new state space is to overcome the metric issue by focusing on the
tree structure only.

In what follows, we refer to $(T,c)$ as an \emph{algebraic tree} if $T\not=\emptyset$ is a
set equipped with a \emph{branch point map} $c\colon T^3\to T$ satisfying consistency conditions (see Definition~\ref{Def:algtree}).
Even though algebraic trees can be seen as metric trees where one has ``forgotten'' the metric (i.e.,
equivalence classes of metric trees), the branch point map is defined such that the notions of leaves, branch
points, degree, subtrees, line segments, etc.\ can be formalized without reference to a metric (and
agree with the corresponding notions in the metric tree).
The Aldous diffusion takes values in the new state space $\mathbb{T}$ of (equivalence classes of)
\emph{algebraic measure trees} introduced in \cite{LoehrWinter} (see Section~\ref{S:algebraic} for algebraic
trees as topological spaces and for equivalence classes of algebraic measure trees).
An algebraic measure tree $(T,c,\mu)$ consists of an
algebraic tree $(T,c)$ satisfying a separability condition, together with a probability measure $\mu$ on it
(see Definition~\ref{Def:amt}).
For a notion of convergence in $\mathbb{T}$, we first introduce the \emph{branch point distribution} on $T$,
\begin{equation}\label{e:nu}
    \nu_{(T,c,\mu)}:=\mu^{\otimes 3} \circ c^{-1},
\end{equation}
and then associate an algebraic measure tree $\smallx=(T,c,\mu)\in \T$ with the metric measure tree
$(T,r_{\mu},\mu)\in \M$. To this end, define the pseudometric
\begin{equation}
\label{e:metric}
  r_\mu(x,y):=\nu_\smallx\([x,y]\)-\tfrac{1}{2}\nu_\smallx\(\{x\}\)-\tfrac{1}{2}\nu_\smallx\(\{y\}\),
\end{equation}
where $x,y\in T$, and $[x,y]$ is the interval (``line segment'') from $x$ to $y$.
We define convergence of the algebraic measure trees in $\T$ as Gromov-weak convergence of these associated
metric measure trees, i.e., we say
\begin{equation}
\label{e:0010}
   (T_N,c_N,\mu_N)_{N\in\N}\mbox{ converges in $\mathbb{T}$\hspace{.5cm} iff\hspace{.5cm} } (T_N,r_{\mu_N},\mu_N)_{N\in\N}\mbox{ converges in $\mathbb{M}$}.
\end{equation}
The space $\mathbb{T}$ equipped with the so-called \emph{branch point distribution distance} Gromov-weak topology
(or, for short, bpdd-Gromov-weak topology) is introduced and further studied in \cite{LoehrWinter}.
Because cladograms are by definition binary, it is for the purpose of the present paper enough to consider the
subspace of $\T$ consisting of binary trees. More precisely, we consider the subspaces
\begin{equation}
\label{e:T2}
   \mathbb{T}_2:=\big\{(T,c,\mu)\in\mathbb{T}:\,\mbox{degrees at most $3$, atoms of $\mu$ only at leaves}\big\}
\end{equation}
of binary trees with no atoms on the skeleton, and
\begin{equation}
\label{e:Tcont}
   \mathbb{T}^{\mathrm{cont}}_2:=\big\{(T,c,\mu)\in\Tbin:\,\mbox{$\mu$ non-atomic}\big\}
\end{equation}
of continuum binary trees. It is shown in \cite[Theorem~3]{LoehrWinter} that both $\mathbb{T}_2$ and
$\mathbb{T}_2^{\mathrm{cont}}$ are compact, which is very convenient for showing tightness of the approximating
processes. Furthermore, on $\Tbin$, we have equivalent formulations of bpdd-Gromov-weak convergence which we
can use to prove our limit statements (see Section~\ref{S:algebraic} for more details).

 Let $\Clad$ denotes the set of $m$-cladograms (see (\ref{e:CCm})). For an algebraic tree $(T,c)$ and
 $\uu=(u_1,\ldots,u_m) \in T^m$, let $\shape(\uu)$ denote the
 $m$\nbd cladogram generated by the points $u_1,...,u_m$ in $(T,c)$ (see Definition~\ref{Def:shapeclado} for a precise definition).
For $m\in\N$ and
$\testtree\in\Clad$, let $\Phi^{m,\mathfrak{t}}$ be the function which sends an algebraic measure tree to the probability that $m$ points sampled independently with $\mu$ generate
the cladogram $\testtree$, i.e.,
\begin{equation}
\label{e:set1}
  \Phi^{m,\mathfrak{t}}(T,c,\mu):=\mu^{\otimes m}\bigl(\shape^{-1}(\mathfrak{t})\bigr),
\end{equation}
where $(T,c,\mu)\in \Tbin$. We refer to $\mu^{\otimes m}\circ \shape^{-1}$ as $m$-sample shape distribution, and
to functions in the linear span of functions of the form \eqref{e:set1} as \emph{shape polynomials}.
One of the main results of \cite{LoehrWinter} is
that $\Phi^{m,\testtree} \in \C_b(\Tbin)$, and moreover, the set of shape polynomials is convergence determining
for measures on $\Taldous$. Therefore, it is a convenient set of test functions.

We characterize the Aldous diffusion analytically as the unique solution of a martingale problem. We use the
following terminology (see Sections~4.3 and 4.4 of \cite{EthierKurtz86}). Let $E$ be a polish space, $B(E)$ be the set of
bounded, measurable, real-valued functions on $E$.

\begin{definition}[Martingale problem]\label{d:mp}
	Let $A \colon \CD(A) \to B(E)$ with $\CD(A)\subseteq B(E)$ be a linear operator, and $P$ a probability measure on $E$.
\begin{enumerate}
	\item
	A \emph{solution of} the $(A,\CD(A),P)$-martingale problem is an $E$-valued, measurable stochastic process $X=(X_t)_{t\ge0}$
		such that $P$ is the law of $X_0$ and, for all\/ $\Phi\in\CD(A)$, the
		process\/ $M:=(M_t)_{t\ge 0}$ given by
		\begin{equation}
		\label{e:martingales}
		   M_t := \Phi(X_t) - \int^t_0 A\Phi(X_s) \,\mathrm{d}s
		\end{equation}
		is a martingale (w.r.t.\ the natural, augmented filtration of $X$).
	\item The $(A, \CD(A), P)$-martingale problem is \emph{well-posed} if there exists a unique (in finite
		dimensional distribution) solution of it.
\end{enumerate}
\end{definition}

If the $(A,\CD(A),P)$-martingale problem is well-posed for every probability measure $P$ on $E$, the solution
$X$ is necessarily a Markov process by \cite[Theorem~4.4.2]{EthierKurtz86}.
We sometimes call the operator $A$ \emph{pre-generator} of $X$, because it is the restriction of the full
generator to $\CD(A)$.
As pre-generator of the Aldous diffusion, we introduce the operator $\AldGen$ acting on functions
of the form \eqref{e:set1} as follows:
\begin{equation}\label{e:set3}
   \AldGen\Phi^{m,\mathfrak{t}}(T,c,\mu)
   :=
   \int\Omega_m\one_{\mathfrak{t}}\(\shape(\uu)\) \,\mu^{\otimes m}(\dx \uu).
\end{equation}
Obviously, $\AldGen$ can be extended linearly to the set of shape polynomials, i.e.\ to
\begin{equation}\label{25_1}
	\CD(\AldGen) := \linspan \bigl\{\,\text{functions $\Phi^{m,\mathfrak{t}}$ of the form
			\eqref{e:set1} with }m\in\N,\, \testtree\in\Clad\bigr\},
\end{equation}
where $\linspan$ denotes the linear span of a set of functions.
Our first main result is the following:
\begin{theorem}[The well-posed martingale problem]\label{T:Aldous}
	For all probability measures\/ $P_0$ on\/ $\Taldous$, the\/ $(\AldGen,\CD(\AldGen),P_0)$-martingale
	problem is well-posed.
	Its unique solution is a Feller process with continuous paths, taking values in the compact state
	space\/ $\Taldous$.  In particular, it is a strong Markov process. Moreover, this solution is ergodic
	with the algebraic measure Brownian CRT as unique invariant distribution.
\end{theorem}

We refer to the process from Theorem~\ref{T:Aldous} as \emph{Aldous diffusion}:

\begin{definition}[Aldous diffusion on binary algebraic measure trees] \label{Def:003}
The unique solution of the $(\AldGen,{\mathcal D}(\AldGen),P_0)$-martingale problem is called \emph{Aldous
diffusion on binary algebraic non-atomic measure trees}, or simply \emph{Aldous diffusion}, started in $P_0$.
\end{definition}

It is important to mention that the Aldous diffusion is dual to the Aldous chain, as for all
$m\in\N$ and $m$-cladograms $\mathfrak{t}$, the Aldous diffusion $X_t=(T_t,c_t,\mu_t)$ started in
$X_0=(T,c,\mu)\in \Taldous$ satisfies
\begin{equation} \label{e:set4}
   \mathbb{E}_{(T,c,\mu)}\Bigl[\mu_t^{\otimes m}\bset{\uu\in T_t^m}{\shape[(T_t,c_t)](\uu)=\testtree} \Bigr]
   =
   \mathbb{E}_{\testtree}\Bigl[\mu^{\otimes m}\bset{\uu\in T^m}{\shape(\uu)={\mathcal T}_t}\Bigr],
\end{equation}
where $({\mathcal T}_t)_{t\ge 0}$ denotes the Aldous chain on $m$-cladograms started in $\testtree$.

The name \emph{Aldous diffusion} is justified by the following convergence result.
Here, we identify the ${\mathfrak{C}_N}$-valued Aldous chain on $N$\nbd cladograms with the $\Tbin$\nbd valued
Markov chain obtained by forgetting the labels of the cladograms and equipping it with the uniform distribution
on the leaves.
\begin{theorem}[Diffusion approximation]\label{T:diffappr}
For each\/ $N\in\N$, let\/ $\smallx_N$ be an\/ $N$-cladogram with the uniform distribution on the leaves.
Assume that\/ $\smallx_N\to\smallx\in\mathbb{T}_2^{\mathrm{cont}}$. Then the Aldous chain\/ $X^N$ starting in\/
$X_0^N=\smallx_N$ converges weakly in Skorokhod path space w.r.t.\ the bpdd-Gromov-weak topology to the Aldous diffusion
starting in\/ $\smallx$.
\end{theorem}

Our last result makes a connection to Aldous's original calculation \eqref{e:FVnegative} of the evolution of the
relative sizes of the three subtrees attached to a fixed branch point until that branch point disappears.
Instead of fixing a branch point in the beginning, we take the average over branch points w.r.t.\ the branch point
distribution \eqref{e:nu}. Our topology on $\Tbin$ turns out to be strong enough for us to use the diffusion
approximation from Theorem~\ref{T:diffappr} to extend the martingale probelm for the Aldous diffusion to the
corresponding test functions.
Thus we can do explicit calculations which show the missing term compensating for the disappearance of branch points.

To state the result, we need some notation. For a branch point $v\in \br(T)$, consider the three subtrees
(components) attached to $v$, and denote by $\Sub_v(u)$ the one containing $u\in T$ with
$u\ne v$ (see \eqref{e:equiv} below for a precise definition).
For $\uu=(u_1,u_2,u_3) \in T^3$, let
\begin{equation}\label{e:eta3}
	\ueta(\uu) := \(\eta_i(\uu)\)_{i=1,2,3} := \( \mu\(\Sub_{c(\uu)}(u_i) \)_{i=1,2,3}
\end{equation}
be the vector of the three masses of the components connected to the branch point $c(\uu)$ of $\uu$.
We consider test functions of the following form, called \emph{mass polynomials of degree $3$}: For
$f\colon [0,1]^3 \to \R$ continuous define
\begin{equation}
\label{e:Phi}
   \Phi^{f}(T,c,\mu)
   :=
   \int f\(\ueta(\uu)\) \,\mu^{\otimes 3}(\mathrm{d}\uu),
\end{equation}
where $(T,c,\mu)\in \Tbin$. One of the main results of \cite{LoehrWinter} is that $\Phi^f \in \C(\Tbin)$.

We can extend the domain of the pre-generator $\AldGenMass$ to the set of those mass polynomials $\Phi^f$ of
degree $3$ with $f\colon[0,1]^3\to\R$ twice continuously differentiable. To this end, consider the
migration operators $\Theta_{i,j} \colon \C^2([0,1]^3) \to \C^1([0,1]^3)$, $i,j\in \{1,2,3\}$, $i\ne j$
\begin{equation}\label{e:migration}
	\Theta_{1,2} f (\ux) := \frac{\one_{x_1 > 0}}{x_1}\( f(0,x_2+x_1,x_3) - f(\ux) \)
				+ \one_{x_1=0} \(\partial_2f(\ux) - \partial_1f(\ux) \),
\end{equation}
and $\Theta_{i,j}f$ defined analogously with the indices $1$ and $2$ replaced by $i$ and $j$, respectively.
Let $e_i=(\delta_{ij})_{j=1,2,3}$ be the $i^{\mathrm{th}}$ unit vector and define
\begin{equation}
\label{e:generator}
\begin{aligned}
   \AldGenMass \Phi^f(T,c,\mu)
	&:= \int_{T^3}\dx\mu^{\otimes3}\,\Bigl( 2\sum_{i,j=1}^3 \eta_i(\delta_{ij} - \eta_j) \partial_{ij}^2 f(\underline\eta)
		+3\sum_{i=1}^3 (1-3\eta_i)\partial_i f(\underline\eta) \\
  &\phantom{{}+\int_{T^3}\dx\mu^{\otimes3}\,\Bigl({}} + \tfrac{1}{2}\sum_{i,j=1,\, i\ne j}^3 \Theta_{i,j}f(\ueta)
		+ \sum_{i=1}^3\(f(e_i)-f(\underline\eta)\)\Bigr).
\end{aligned}
\end{equation}

\begin{theorem}[Extended martingale problem for subtree masses]\label{t:massgen}
Let\/ $X=(X_t)_{t\ge 0}$ be the Aldous diffusion on\/ $\Taldous$.
Then for all test functions\/ $\Phi^f$ of the form \eqref{e:Phi} with\/ $f\in\C^2\([0,1]^3\)$, the
process\/ $M^f:=(M^f_t)_{t\ge 0}$ given by
\begin{equation}
\label{e:071}
   M_t^{f}:=\Phi^f(X_t)-\Phi^f(X_0) - \int^t_0\AldGenMass\Phi^f(X_s)\,\mathrm{d}s
\end{equation}
is a martingale.
\end{theorem}

\smallskip\noindent{\bf Related work.}
We note that a construction related to the Aldous diffusion has been recently established independently in a
sequence of papers
\cite{FormanPalRizzoloWinkel2016,FormanPalRizzoloWinkel2018a,FormanPalRizzoloWinkel2018b,FormanPalRizzoloWinkel2018c}.
A discussion of the differences is therefore in order. Their construction was first sketched in
\cite{Pal2011}. Pal suggests to first take a finite number of branch points, consider the cladogram spanned by them, and decompose the lines connecting any two neighboring branch points of this cladogram into subtrees. Then study the suitably rescaled subtree masses as partitions
of an interval of random length while relaxing the constrain that the total number of vertices must be preserved by letting removing and inserting of edges happen independently. When applying the time change which reverses the described Poissonization, on the proposed time scale the masses converge to an evolving interval partition described by a family of diffusions indexed by $\N$. However, if this
chain runs,
then the mass branching off one of the external edges of the cladogram gets exhausted. When this happens, the dynamics breaks
down, and one needs to find a slightly different set of branch points to
proceed.
To resolve the problem of disappearing vertices \cite{FormanPalRizzoloWinkel2018a} suggests a
smart way of swapping labels of the cladograms in such a way that
the resulting dynamics preserves stationarity when one starts from the uniform
distribution.

Our construction is related in spirit but differs in some important aspects.
First, rather than sampling cladograms and describing their dynamics under the Aldous chain, we describe the
behavior of the \emph{average} of the quantities of interest over uniformly sampled cladograms. This allows us to give a nice characterization of the Aldous diffusion as a unique solution of some martingale problem. As a consequence, we do not require the initial distribution for the Aldous diffusion to be uniform but can rather let it start
in any deterministic continuum tree. We can show that the Aldous chain converges weakly in path space to the
Aldous diffusion, and that the latter is a Feller process.
Note that \cite{FormanPalRizzoloWinkel2018c} never states explicitly that the \Rtree-valued diffusion
constructed is a strong Markov process.
Furthermore, we are also able to state a duality relation, which
allows us to conclude convergence to the uniform distribution for all starting points as time tends to infinity.
In \cite{LoehrWinter}, we put some effort in establishing with the space of algebraic measure trees a new state space
and invested in a detailed study of topological aspects. As a result, we obtained equivalent formulations of our
notion of convergence on the subspace of binary trees which made martingale convergence statements very much
straightforward. Finally, the framework provided is not restricted to the construction of the continuum limit of the Aldous Markov chain. It can also be applied to other (not necessarily symmetric) sampling consistent tree dynamics. For example, in \cite{Nussbaumer} a tree-valued dynamics is constructed which has the algebraic measure Kingman tree as its stationary distribution.

Other approaches of encoding relatives of binary algebraic trees can be found in
\cite{Forman2018} and \cite{EvansGruebelWakolbinger2017}. The \emph{R\'emy chain} considered in
\cite{EvansGruebelWakolbinger2017} is a Markov chain of growing (ordered) trees that is somewhat related to the
Aldous chain: it is the process obtained by successively inserting new leaves at randomly chosen edges without
removing a leaf before.

\smallskip\noindent{\bf Outline. }
The rest of the paper is organized as follows.
In Section~\ref{S:algebraic}, we introduce our state space of algebraic measure trees and recall its most
important properties from \cite{LoehrWinter}.

In Section~\ref{S:generator}, we show tightness of the Aldous chains and existence of solutions of the
martingale problem from Theorem~\ref{T:Aldous}. We do so by using and proving uniform convergence of (pre-)generators.
In Section~\ref{S:duality}, we obtain the duality for the Aldous diffusion (Proposition~\ref{P:duality}), and use
it to show uniqueness of solutions of the martingale problem.
In Section~\ref{S:longterm}, we show that the Aldous diffusion has a unique invariant measure, namely the
algebraic measure Brownian CRT, and that the Aldous diffusion converges to it in law as time goes to infinity
(Proposition~\ref{P:Aldlong}). We also finish the proof of Theorems~\ref{T:Aldous} and \ref{T:diffappr}.
In Section~\ref{S:subtreemass}, we prove Theorem~\ref{t:massgen} and apply it to calculate the annealed average
distance of two points in the Brownian CRT with respect to our intrinsic metric.

\section{The state space of binary, algebraic measure trees}
\label{S:algebraic}
In this section we introduce the state space. The goal is
to overcome the metric issue raised in the introduction by focusing on the algebraic tree structure only.
We encode the cladograms as binary, algebraic trees, and use the space of these trees together with the
bpdd-Gromov-weak topology studied in \cite{LoehrWinter}. All proofs can be found there.

\begin{definition}[Algebraic tree]\label{Def:algtree}
An \emph{algebraic tree} is a non-empty set $T$ together with a symmetric map $c\colon T^3\to T$ satisfying the following:
\begin{enumerate}[\quad\bf (1pc)]\setcounter{enumi}{1}
  \item  For all $x_1,x_2\in T$, $\,c(x_1,x_2,x_2)=x_2$.
  \item  For all $x_1,x_2,x_3\in T$, $\,c\(x_1,x_2,c(x_1,x_2,x_3)\)=c(x_1,x_2,x_3)$.
  \item  For all $x_1,x_2,x_3,x_4\in T$, \begin{equation}
      \label{e:4pt}
      c(x_1,x_2,x_3)\in\bigl\{c(x_1,x_2,x_4),\,c(x_1,x_3,x_4),\,c(x_2,x_3,x_4)\bigr\}.
      \end{equation}
\end{enumerate}
We refer to the map $c$ as \emph{branch point map}. A \emph{tree isomorphism} between two algebraic trees
$(T_i,c_i)$, $i=1,2$, is a bijective map $\phi\colon T_1 \to T_2$ with $\phi\(c_1(x_1,x_2,x_3)\)=c_2\(\phi(x_1),
\phi(x_2), \phi(x_3)\)$ for all $x_1,x_2,x_3\in T_1$.
\end{definition}

For each point $x\in T$, we define an equivalence relation $\sim_x$ on $T\setminus\{x\}$ such that for all
$y,z\in T\setminus\{x\}$,
$y\sim_x z$
  iff $c(x,y,z)\not =x$.
For $y\in T\setminus\{x\}$, we denote by
\begin{equation}\label{e:equiv}
  \Sub_x(y):=\big\{z\in T\setminus\{x\}:\,z\sim_x y\big\}
\end{equation}
the equivalence class w.r.t.\ $x\in T$ which contains $y$. We also call $\Sub_x(y)$ the \emph{component} of
$T\setminus\{x\}$ containing $y$.
An algebraic tree $(T,c)$ allows for all kinds of notions which capture the tree structure, e.g.,
\begin{itemize}
\item we say that $S\subseteq T$ is a \emph{subtree} of $T$ iff $c(S^3)=S$,
\item we call the number of components of $T\setminus \{x\}$ the \emph{degree} of $x\in T$ and write
	$\deg(x)=\#\bset{\Sub_x(y)}{y\in T\setminus\{x\}}$,
\item we say that $u\in T$ is a \emph{leaf} iff $\deg(u)=1$, and write $\lf(T)$ for the set of leaves,
\item we say that $v\in T$ is a \emph{branch point} iff $\deg(v)\ge 3$, or equivalently, $v=c(x_1,x_2,x_3)$ for some
	$x_1,x_2,x_3\in T\setminus\{v\}$, and write $\br(T)$ for the set of branch points,
\item we write $[x,y]$, $x,y\in T$, for the \emph{interval}
\begin{equation}
\label{e:interval}
  [x,y]:=\bset{z\in T}{c(x,y,z)=z},
\end{equation}
\item and we say that $\{x,y\}$ is an edge iff $x\ne y$ and $[x,y]=\{x,y\}$.
\end{itemize}

There is a natural topology on a given algebraic tree, namely the \emph{component topology} generated by the set of all components
$\Sub_x(y)$ as defined in \eqref{e:equiv} with $x\ne y$, $x,y\in T$.
In what follows we refer to an algebraic tree $(T,c)$ as \emph{order separable} if it is separable w.r.t.\ this
topology and has at most countably many edges.
We further equip order separable algebraic trees with a probability measure on the Borel $\sigma$\nbd algebra
${\mathcal B}(T)$ of the component topology.
This so-called \emph{sampling measure} allows to sample vertices from the tree.
\begin{definition}[Algebraic measure trees]\label{Def:amt}
A (separable) \emph{algebraic measure tree} $(T,c,\mu)$ is an order separable algebraic tree $(T,c)$
		together with a probability measure $\mu$ on ${\mathcal B}(T)$.
\end{definition}

In what follows we call two algebraic measure trees $(T_i,c_i,\mu_i)$, $i=1,2$, \emph{equivalent} if there exist
subtrees $S_i\subseteq T_i$ with $\mu_i(S_i)=1$,  $i=1,2$, and a measure preserving tree isomorphism $\phi$ from
$S_1$ onto $S_2$, i.e.,
$c_2(\phi(x),\phi(y),\phi(z))=\phi(c_1(x,y,z))$ for all $x,y,z\in S_1$, and $\mu_1\circ\phi^{-1}=\mu_2$.
We define
\begin{equation}
\label{e:004}
    \T := \mbox{set of equivalence classes of algebraic measure trees}.
\end{equation}
With a slight abuse of notation, we will write $\smallx=(T,c,\mu)$ for the algebraic tree as well as the
equivalence class.
Note that $\deg(x),\,\lf(T),\,\edge(T),\ldots$ are properties of the particular representative
and not preserved under equivalence, because we do not require the whole trees to be isomorphic.
For instance, every equivalence class contains a representative without edges (informally, we can replace edges
by line segments carrying no measure). 

Recall the branch point distribution $\nu=\nu_\smallx=\mu^{\otimes 3}\circ c^{-1}$ and the pseudometric $r_\mu$ from
Equations \eqref{e:nu} and \eqref{e:metric}, respectively.
For every equivalence class of algebraic measure trees, a representative $(T,c,\mu)$ can be chosen such that
$r_\mu$ is a metric (by identifying points of distance zero in any representative).
One can check that in this case, $r_\mu$ induces the component topology, and $(T,r_\mu)$ is a separable metric
tree in the sense of \cite{AthreyaLohrWinter2017} (i.e., isometric to a subset of an \Rtree\ containing all branch points)
satisfying for all $x,y,z\in T$,
\begin{equation}
\label{e:gener}
  [x,y]_{r_\mu}\cap[x,z]_{r_\mu}\cap[y,z]_{r_\mu}=\big\{c(x,y,z)\big\},
\end{equation}
where $[x,y]_{r_\mu}=\set{v\in T}{r_\mu(x,y)=r_\mu(x,v)+r_\mu(v,y)}$ denotes the interval in $(T,r_\mu)$.
In particular, $[x,y]_{r_\mu}=[x,y]$.
Note that any point which carries positive mass is an isolated point in the metric space $(T,r_\mu)$.

As in any metric tree, we can define for a fixed reference point (root) $\rho\in T$ a unique measure
$\ell^{(T,c,\mu,\rho)}$ on $(T,{\mathcal B}(T))$
which is characterized by the two properties $\ell^{(T,c,\mu,\rho)}((\rho,y]):=r_\mu(\rho,y)$ and
$\ell^{(T,c,\mu,\rho)}\(\lf(T)\setminus \at(\mu)\)=0$,
where $\at(\mu)$ denotes the set of atoms of $\mu$.
The measure $\ell^{(T,c,\mu,\rho)}$ is referred to as \emph{length measure} w.r.t.\ $\rho$. Note
that it depends on the choice of the distinguished point $\rho$.
However, the total mass of the length measure does not depend on the choice of $\rho$ and equals
\begin{equation}
\label{totallength}
  \|\ell^{(T,c,\mu,\rho)}\|:=\ell^{(T,c,\mu,\rho)}\(T\)
  =\tfrac{1}{2}\int_T\deg(v)\;\nu(\mathrm{d}v).
\end{equation}

We define convergence in $\T$ as follows.
\begin{definition}[Bpdd-Gromov-weak topology]
\label{Def:001} We say that a sequence $(\smallx_n)_{n\in\N}$ of (equivalence classes of) algebraic
measure trees $\smallx_n=(T_n,c_n,\mu_n)\in \T$ converges \emph{branch point distribution distance
Gromov-weakly} (\emph{bpdd-Gromov-weakly}) to the algebraic measure tree $(T,c,\mu)\in \T$ iff
the sequence $(\tilde{\smallx}_n)_{n\in\N}$ of (equivalence classes of) metric measure trees
$\tilde{\smallx}_n:=(T_n,r_{\mu_n},\mu_n)\in \M$ converges to the metric measure tree $(T,r_\mu,\mu)\in\M$
Gromov-weakly, i.e., if for $U_1^n, U_2^n, ...$ independent and $\mu_n$-distributed, and $U_1, U_2, ...$ independent and $\mu$-distributed, for all $m\in\N$,
\begin{equation}
  \(r_{\mu_n}(U^n_i,U^n_j)\)_{1\le i<j\le m} \;\Tno\; \(r_\mu(U_i,U_j)\)_{1\le i<j\le m}.
\end{equation}
\end{definition}

In this paper we are only considering binary algebraic measure trees with the property that the measure has
atoms only (if at all) on the leaves of the tree, i.e.\ the subspace of $\T$ given by
\begin{equation}
\label{e:binary}
   \Tbin=\bset{(T,c,\mu)\in \T}{\deg(v)\le 3\,\forall v\in T,\; \at(\mu)\subseteq\lf(T)},
\end{equation}
(compare \eqref{e:T2}). Even though the equivalence class $\smallx\in\Tbin$ contains algebraic measure trees which are
not binary, we will implicitly assume that the chosen representative $(T,c,\mu)$ satisfies $\deg(v)\le 3$.
In this subspace, it turns out that bpdd-Gromov-weak convergence is equivalent to another very
useful notion of convergence, namely the so-called sample shape convergence, which we introduce next.

\begin{definition}[\nclad]\label{Def:mclado}
For $m\in\N$, an \emph{\nclad} is a binary, finite tree $(C,c)$ consisting only of leaves and branch points,
together with a surjective labelling map $\zeta\colon \{1,...,m\} \to \lf(C)$.
\end{definition}

Note that an \nclad\ has at most $m$ leaves (and $m-2$ branch points), but can have less if a leaf has multiple
labels. An \nclad\ $(C,c,\zeta)$ is an $m$\nbd cladogram if and only if $\zeta$ is injective.
We call two \nclads\ $(C_1,c_1,\zeta_1)$ and $(C_2,c_2,\zeta_2)$ \emph{isomorphic} if there exists a tree
isomorphism $\phi$ from $(C_1,c_1)$ onto $(C_2,c_2)$ such that $\zeta_2=\phi \circ \zeta_1$.
Furthermore, we denote the sets of isomorphism classes of \nclad\ and $m$\nbd cladograms by $\mlClad$ and $\Clad$,
respectively, i.e.,
\begin{equation}
\label{e:CCm}
\begin{aligned}
	\mlClad &:=\bigl\{\text{isomorphism classes of \nclads}\bigr\}
\end{aligned}
\end{equation}
and
\begin{equation}
\label{e:CC}
\begin{aligned}
	\Clad &:=\bset{(C,c,\zeta) \in \mlClad}{\zeta \text{ injective}}.
\end{aligned}
\end{equation}

\begin{definition}[The shape function]\label{Def:shapeclado}
For a binary algebraic tree $(T,c)$, $m\in\N$, and $u_1,...,u_m\in T\setminus \br(T)$ (not necessarily
distinct), there exists a unique (up to isomorphism) \nclad\
\begin{equation}\label{e:shape}
   \shape(u_1,...,u_m)=(C,c_C,\zeta)
\end{equation}
with $\lf(C)=\{u_1,...,u_m\}$ and $\zeta(i)=u_i$, such that the identity on $\lf(C)$ extends to a tree
homomorphism $\pi$ from $C$ onto $c\(\{u_1,...,u_m\}^3\)$, i.e., for all $i,j,k=1,...,m$,
\begin{equation}
   \pi\(c_C(u_i,u_j,u_k)\)=c(u_i,u_j,u_k).
\end{equation}
We refer to $\shape(u_1,...,u_m)\in\mlClad$ as the
\emph{shape} of $u_1,...,u_m$ in $(T,c)$.
\end{definition}

\xymatfig{f:shape}
	{&u_1\xyedge[d] &&&&&&&&\\
	 \xyedge[r]&\node\xyedge[dr] &     &   &\node\xyedge[ul]\xyedge[ur]       &      &&&&&   u_1\xyedge[dr] & & u_3\xyedge[d] & \\
	     &&\node\xyedge[r]&u_3\xyedge[r]&\node\xyedge[r]\xyedge[u]&u_4&&&&&     &\node\xyedge[r]&\node\xyedge[r]&u_4\\
	 &u_2\xyedge[ur] &     &   &  &           &&&&&   u_2\xyedge[ur] & & &}
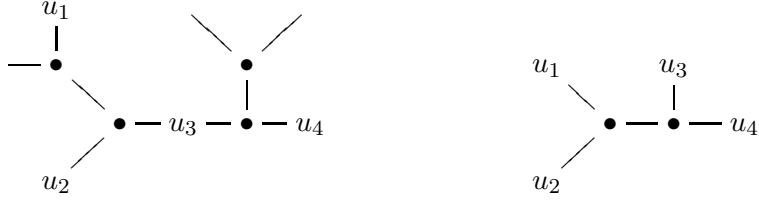
{A tree $T$ and the shape $\shape(u_1,u_2,u_3,u_4)$. The cladogram is not isomorphic to the subtree
$c(\{u_1,u_2,u_3,u_4\}^3)$ because $u_3\in \openint{u_1}{u_4}$.}

\begin{definition}[Sample shape convergence]
We say that a sequence $(\smallx_n)_{n\in\N}$ of (equivalence classes of) binary algebraic measure trees $(T_n,c_n,\mu_n)$ \emph{converges in sample shape} to the (equivalence class of the) algebraic measure tree $(T,c,\mu)$ iff for $U_1^n, U_2^n, ...$ independent and $\mu_n$-distributed, and $U_1, U_2, ...$ independent and $\mu$-distributed, for all $m\in\N$,
\begin{equation}
  \shape\(U^n_1,...,U^n_m\)\Tno \shape\(U_1,...,U_m\).
\end{equation}
\end{definition}

To be later in a position to recover the calculations of Aldous and others concerning the dynamics of subtree masses, we introduce yet another notion of convergence.

\begin{definition}[Sample subtree mass convergence]
We say that a sequence $(\smallx_n)_{n\in\N}$ of (equivalence classes of) algebraic measure trees $(T_n,c_n,\mu_n)$ \emph{converges in sample subtree mass} to the (equivalence class of the) algebraic measure tree $(T,c,\mu)$ iff for $U^n_1, U^n_2, ...$ independent and $\mu_n$-distributed, and $U_1, U_2, ...$ independent and $\mu$-distributed, for all $m\in\N$,
\begin{equation}
  \(\mu_n(\Sub_{c_n(U^n_i,U^n_j,U^n_k)}(U^n_i))\)_{i,j,k=1,...,m} \;\Tno\;
  	\bigl(\mu(\Sub_{c(U_i,U_j,U_k)}(U_i))\bigr)_{i,j,k=1,...,m}.
\end{equation}
\end{definition}

The following results are crucial for the construction of Aldous diffusion and stated in \cite[Proposition~2.32, Theorem~3, Corollary~5.21]{LoehrWinter}.
On $\Tbin$, all of the above notions of convergence are equivalent.
By \eqref{totallength}, the total length of binary algebraic measure trees is uniformly bounded by $3$, and one can show that the space $\Tbin$ is compact.

\begin{proposition} \label{P:conveq}
	Let\/ $(\smallx_N=(T_N,c_N,\mu_N))_{N\in\N}$ and\/ $\smallx=(T,c,\mu)$ be in\/ $\Tbin$.
The following are equivalent:
\begin{enumerate}[(1)]
	\item $\smallx_N\tno\smallx$ w.r.t.\ sample shape convergence.
\item For all\/ $m\in\N$ and\/ $\mathfrak{t}\in\mlClad$,
\begin{equation}
\label{e:005}
   \mu_N^{\otimes m}\big\{(u_1,...,u_m):\,\mathfrak{s}_{(T_N,c_N)}\(\uu\)=\mathfrak{t}\big\}
   \tno
   \mu^{\otimes m}\big\{(u_1,...,u_m):\,\shape\(\uu\)=\mathfrak{t}\big\}.
\end{equation}
\item $\smallx_N\tno\smallx$ Gromov-weakly w.r.t.\ the branch point distribution distance.
\item For all\/ $m\in\N$ and\/ $\phi\in{\mathcal C}_b(\R_+^{m\times m})$,
\begin{equation}
\label{e:008}
   \int\mu_N^{\otimes m}(\mathrm{d}\uu)\,\phi\((r_{\mu_N}(u_i,u_j))_{1\le i,j\le m}\)
   \tno
   \int\mu^{\otimes m}(\mathrm{d}\uu)\,\phi\((r_\mu(u_i,u_j))_{1\le i,j\le m}\).
\end{equation}
\item $\smallx_N\tno\smallx$ w.r.t.\ sample subtree mass convergence.
\item For all\/ $m\in\N$ with\/ $m\ge 3$ and\/ $f\in{\mathcal C}_b([0,1]^{m^3})$,
\begin{equation}
\label{e:009}
\begin{aligned}
 \MoveEqLeft \int\mu_N^{\otimes
 m}(\mathrm{d}\uu)\,f\Bigl(\(\mu_N(\Sub_{c_N(u_i,u_j,u_k)}(u_i))\)_{i,j,k=1,\ldots,m}\Bigr)
   \\
   &\tno
   \int\mu^{\otimes m}(\mathrm{d}\uu)\,f\Bigl(\(\mu(\Sub_{c(u_i,u_j,u_k)}(u_i))\)_{i,j,k=1,\ldots,m}\Bigr)
\end{aligned}
\end{equation}
\end{enumerate}
\end{proposition}

In what follows, we will need the following two subspaces of $\Tbin$. Let for each $N\in\N$,
\begin{equation}
\label{e:T2N}
	\Tbin^N:=\bset{(T,c,\mu)\in\Tbin}{\#\lf(T)=N\text{ and }\mu=\tfrac{1}{N}\sum_{u\in\lf(T)}\delta_u},
\end{equation}
and
\begin{equation}
\label{e:Taldous}
   \Taldous:=\bset{(T,c,\mu)\in\Tbin}{\at(\mu)=\emptyset}.
\end{equation}
The Aldous chain on $N$-cladograms is naturally defined on $\Tbin^N$: for $\smallx \in \Tbin^N$, there is a
unique (up to measure preserving tree isomorphism) minimal representative $(T,c,\mu)$ of $\smallx$
(i.e.\ no subset with the restrictions of $c$ and $\mu$ is an algebraic measure tree) with $2N-2$ vertices and
$2N-3$ edges. We identify $\smallx\in \Tbin^N$ with this minimal representative and interpret it as
``$N$-caldogram without labels'' with uniform distribution on the leaves. We define the Aldous chain on
$\Tbin^N$ in the same way as the one on $\Clad[N]$, via its generator $\Omega_N$ in \eqref{e:genN}.
With a slight abuse of notation, we use the same notation for the generators of the $\Tbin^N$-valued and of the $\Clad[N]$-valued chain. 
We will define the Aldous diffusion on the
space $\Taldous$ in view of the following approximation result:
\begin{proposition}[Approximations with $\Tbin^N$]\label{p:001}
Let\/ $\smallx\in \Tbin$. Then\/
$\smallx\in\mathbb{T}_2^{\mathrm{cont}}$ if and only if there exists for each\/
$N\in\N$ an\/ $\smallx_N\in\mathbb{T}_2^N$ such that\/ $\smallx_N\to\smallx$ in one (and thus all) of the
equivalent notions of convergence on\/ $\mathbb{T}_2$ given above.
\end{proposition}

\begin{proposition}[Compactness and metrizability]\label{P:compact}
$\Tbin$ is a compact, metrizable space.
Both\/ $\Tbin^N$ and\/ $\Taldous$ are closed subspaces of\/ $\mathbb{T}_2$, and thus compact as well.
\end{proposition}
To deal with the Aldous chain and diffusions, it is convenient to introduce the following set of test functions on
$\Tbin$.

\begin{definition}[Shape polynomials]\label{Def:002}
A \emph{shape polynomial} is a linear combination of functions $\Phi^{m,\mathfrak{t}}\colon \Tbin\to\R$ of the form
\begin{equation}
\label{s:009}
   \Phi^{m,\mathfrak{t}}(\smallx):=\mu^{\otimes m}\(\shape^{-1}(\mathfrak{t})\),
\end{equation}
where $\smallx=(T,c,\mu)$, $m\in\N$ and $\mathfrak{t}\in\mlClad$.
Let\/ $\sPol$ be the set of all shape polynomials.
\end{definition}

Apart from its combinatorial nature, the usfulness of shape polynomials stems from the fact that every
real-valued continuous function on $\Tbin$ can be approximated by them.

\begin{lemma}\label{l:polalgebra}
	The set\/ $\sPol$ of shape polynomials is a uniformly dense sub-algebra of\/ $\C(\Tbin)$.
\end{lemma}
\begin{proof}
	It is immediate from Propositions~\ref{P:conveq} and \ref{P:compact} that $\sPol$ is contained in the
	space $\C(\Tbin)$ of continuous functions and separates the points of $\Tbin$.
	To see that $\sPol$ is multiplicatively closed, consider for $k,m\in\N$, $k<m$ the projection $\pi_{m,k} \colon
	\mlClad \to \mlClad[k]$, mapping $\testtree=(C,c,\zeta)\in \mlClad$ to the subcladogram spanned by
	$\zeta(1),\ldots,\zeta(k)$, as well as the projection $\tilde\pi_{m,k} \colon \mlClad\to\mlClad[m-k]$,
	mapping to the subcladogram spanned by $\zeta(k+1),\ldots,\zeta(m)$ (relabelled to have labels in
	$\{1,\ldots,m-k\}$). Because $\mu$ is a probability and the product measure is used in \eqref{s:009}, 
	we have for $m,n\in\N$, $\testtree\in\mlClad$, $\tilde\testtree\in\mlClad[n]$
\begin{equation}
	\Phi^{m,\testtree} \cdot \Phi^{n,\tilde\testtree}
	  = \sum_{\testtree'\in\pi_{n+m,m}^{-1}(\testtree) \cap \tilde\pi_{n+m,n}^{-1}(\tilde\testtree)}
			\Phi^{m+n,\testtree'} \in \sPol.
\end{equation}
	Because $\Tbin$ is compact by Proposition~\ref{P:compact}, $\sPol$ is dense in $\C(\Tbin)$ by the
	Stone-Weierstrass theorem.
\end{proof}

Consider $m\in\N$ and $\testtree=(C,c,\zeta)\in\mlClad \setminus \Clad$, i.e.\ there is at least one leaf in
the \nclad\ $\testtree$ with multiple labels. Then $\shape(u_1,\ldots,u_m)=\testtree$ implies that the $u_1,\ldots,u_m$
are not distinct. Hence $\Phi^{m,\testtree}(\smallx) = 0$ for all $\smallx\in\Taldous$. This is in fact the
reason why we restricted the domain of the pre-generator of the Aldous diffusion $\CD(\AldGen)$ to shape
polynomials using $m$\nbd cladograms instead of $m$-labelled cladograms (see~\eqref{25_1}).
Note that the set of restrictions to $\Taldous$ of functions in $\CD(\AldGen)$ is dense in $\C(\Taldous)$.

\section{Convergence of generators, tightness and existence}
\label{S:generator}
In this section we prepare the proofs of our main results by showing the uniform convergence of the
generators of the discrete chains to the pre-generator $(\CD(\AldGen),\AldGen)$, and deduce tightness of the
Aldous chains (provided tightness of initial conditions) as well as existence of solutions of the limiting
martingale problem by general theory. We also obtain continuous paths of all limit processes.

A first simple observation about the pre-generator is that it maps $\CD(\AldGen)$ into itself.
\begin{lemma}\label{l:almostFeller}
	For every\/ $\Phi\in \CD(\AldGen)$, we have\/ $\AldGen\Phi \in \CD(\AldGen)$. In particular,
	$(\Phi, \AldGen\Phi) \in \C(\Tbin)\times \C(\Tbin)$.
\end{lemma}
\begin{proof}
	Both $\Phi$ and $\AldGen\Phi$ are shape polynomials, hence continuous by definition of sample shape
	convergence.
\end{proof}

For $N\in\N$, recall from \eqref{e:genN} the generator $\Omega_N$ of the Aldous chain on $N$-cladograms
and from \eqref{e:T2N} the space $\mathbb{T}_2^N$
of algebraic measure trees with $N$ leaves and uniform distribution on the leaves.

\begin{proposition}[Convergence of generators]\label{P:generator}
For all\/ $\Phi\in{\mathcal D}(\AldGen)$, we have
\begin{equation}\label{e:011}
  \lim_{N\to\infty}\sup_{\smallx\in\mathbb{T}_2^{N}}\big|\Omega_N\Phi(\smallx)-\AldGen\Phi(\smallx)\big|
  = 0.
\end{equation}
\end{proposition}

\begin{proof}
Consider $\Phi \in \CD(\AldGen)$. By linearity, we may assume w.l.o.g.\ that $\Phi=\Phi^{m,\testtree}$ for some
$m\in\N$ and $\testtree\in\Clad$. In particular, $\mathfrak{t}$ is such that no leaf has multiple labels, and
consequently for $\uu\in T^m$, $\shape(\uu)=\mathfrak{t}$ implies that $u_1, ..., u_m$ are
distinct.

Fix $N\in\N$ and $\smallx \in \Tbin^N$, and let $(T,c,\mu)$ be the unique (up to measure preserving tree
isomorphism) minimal representative (i.e.\ $\#T=2N-2$). Then $\#\lf(T)=N$ and $\#\edge(T)=2N-3$).
In the following we abbreviate the inverse numbers of leaves and edges respectively by
\begin{equation}
\label{e:epsN}
   \eps=\eps_N:=\tfrac1N,\quad\mbox{ and }\quad
   \delta=\delta_N:=\tfrac{1}{2N-3}.
\end{equation}
We extend the algebraic tree to allow for potential new branch points (due to inserting an edge) and new leaves.
To this end, for every edge $e\in\edge(T)$, we introduce two additional points $x_e,y_e$. Informally, $x_e$ is
inserted in the middle of $e$, and $y_e$ is attached to $x_e$ as a leaf. More precisely, we consider
\begin{equation}
\label{e:Tbar}
   \Tb=T \uplus\biguplus_{e\in \edge(T)}\{x_e,y_e\},
\end{equation}
and extend $c$ to $\cb\colon \Tb^3\to \Tb$ which is uniquely defined by the following.
$(\Tb,\cb)$ is an algebraic tree such that for $e=\{u,v\}\in\edge(T)$, we have $x_e\in [u,v]$ in $(\Tb,\cb)$, and
\begin{equation}
\label{e:cbar}
   \cb(y_e,x_e,z)=x_e \quad\forall z\in \Tb\setminus\{y_e\}.
\end{equation}
The construction is illustrated in Figure~\ref{f:Tbar}. Note that $(\Tb,\cb,\mu)$ is a binary algebraic measure
tree equivalent to $(T,c,\mu)$.
\comment{\xymatrixrowsep{0.5pc}
\xymatdoublefig[0.5pc]{f:Tbar}{
	\node\xyedge[ddrr] &&                  && \node\xyedge[dd] &&       && \node \\\\
	                   && \node\xyedge[rr] && \node\xyedge[rr] && \node\xyedge[uurr] && \\\\
	\node\xyedge[uurr]^{e} &&                  &&                  &&      && \node\xyedge[uull]
}{
	\node\xyedge[dr] && \Bnode\xyedge[dl]  && \node\xyedge[d] &&       && \node \\
	& \Bnode\xyedge[dr] && &\Bnode\xyedge[d]&\Bnode\xyedge[l] &&\Bnode\xyedge[ur]\xyedge[dr] & \\
	  \Bnode && \node\xyedge[r] &\Bnode\xyedge[r]& \node\xyedge[r] &\Bnode\xyedge[r]& \node\xyedge[ur] &&\Bnode \\
	& \Bnode\xyedge[ur]\xyedge[ul]^(1){y_e} &&\Bnode\xyedge[u] && \Bnode\xyedge[u]&& \Bnode\xyedge[ul] & \\
	\node\xyedge[ur]_(1){x_e} &&                  &&                  && \Bnode\xyedge[ur] && \node\xyedge[ul]
}{A finite algebraic tree $(T,c)$ and the extended tree $(\Tb,\cb)$.}}
{\xymatrixrowsep{0.5pc}\xymatrixcolsep{0.72pc}
\xymatfig{f:Tbar}{
	\node\xyedge[ddrr] &&                  && \node\xyedge[dd] &&       && \node 		&&&&	\node\xyedge[dr] && \Bnode\xyedge[dl]  && \node\xyedge[d] &&       && \node \\
                           &&&&&&&&                                                             &&&&        & \Bnode\xyedge[dr] && &\Bnode\xyedge[d]&\Bnode\xyedge[l] &&\Bnode\xyedge[ur]\xyedge[dr] & \\
	                   && \node\xyedge[rr] && \node\xyedge[rr] && \node\xyedge[uurr] &&     &&&&          \Bnode && \node\xyedge[r] &\Bnode\xyedge[r]& \node\xyedge[r] &\Bnode\xyedge[r]& \node\xyedge[ur] &&\Bnode \\
                           &&&&&&&&                                                             &&&&        & \Bnode\xyedge[ur]\xyedge[ul]^(1){y_e} &&\Bnode\xyedge[u] && \Bnode\xyedge[u]&& \Bnode\xyedge[ul] & \\
	\node\xyedge[uurr]^{e} &&              &&                  &&      && \node\xyedge[uull]&&&&        \node\xyedge[ur]_(1){x_e} &&                  &&                  && \Bnode\xyedge[ur] && \node\xyedge[ul]
}{A finite algebraic tree $(T,c)$ and the extended tree $(\Tb,\cb)$.}}

For $k\in\{1,...,m\}$ and $x\in \Tb$, let $\theta_{k,x}\colon T^m\to \Tb^m$ be the \emph{replacement operator} which
replaces the $k^{\mathrm{th}}$-coordinate by $x$, i.e.,
\begin{equation}
\label{e:thetakx}
	\theta_{k,x}(u_1,...,u_m) = (u_1,...,u_{k-1},x,u_{k+1},..., u_m).
\end{equation}
For $z=(x,e)\in \lf(T)\times\edge(T)$, let
\begin{equation}
\label{e:012}
   \smallx_z:=(\Tb,\cb,\mu+\eps\delta_{y_e}-\eps\delta_x)
\end{equation}
be the binary algebraic measure tree after the Aldous move with $z$.
The difference between sampling with the new and old measure is given by
\begin{equation}\label{e:014}
\begin{aligned}
  \MoveEqLeft \(\mu+\eps\delta_{y_e}-\eps\delta_x\)^{\otimes m}-\mu^{\otimes m} \\
    &= \begin{aligned}[t]
 	& \eps\sum_{k=1}^m\mu^{\otimes (k-1)}\otimes\(\delta_{y_e}-\delta_x\)\otimes\mu^{\otimes (m-k)} \\
 	& +\eps^2\sum_{1\le k<j\le m}\mu^{\otimes (k-1)}\otimes\(\delta_{y_e}-\delta_x\)
		\otimes\mu^{\otimes (j-k-1)}\otimes\(\delta_{y_e}-\delta_x\)\otimes\mu^{\otimes (m-j)} + \tilde{\mu}
 	\end{aligned}\\
    &= \eps\sum_{k=1}^m\(\mu^{\otimes m}\circ\theta^{-1}_{k,y_e}-\mu^{\otimes m} \circ \theta^{-1}_{k,x}\)
    	-\eps^2\sum_{j,k=1,\,j\ne k}^m\mu^{\otimes m}\circ\theta^{-1}_{k,y_e}\circ\theta^{-1}_{j,x}+\tilde{\mu}',
\end{aligned}
\end{equation}
where $\mut, \mut'$ are signed measures on $\Tb^m$ vanishing on $\set{(u_1,\ldots,u_m)}{u_1,\ldots,u_m \text{ distinct}}$.
Thus
\begin{equation}
\label{e:Omega0}
\begin{aligned}
	\Omega_N \Phi^{m,\mathfrak{t}}(\smallx)
  &=
    \sum_{z\in\lf(T)\times\edge(T)}\(\Phi^{m,\mathfrak{t}}(\smallx_z)-\Phi^{m,\mathfrak{t}}(\smallx)\)
		= \sum_{k=1}^m A_k - \!\sum_{j,k=1,\, j\ne k}^m B_{k,j},
\end{aligned}
\end{equation}
with
\begin{align}
\label{e:Ak}
A_k &:= \eps \sum_{(x,e) \in \lf(T)\times\edge(T)}
	\int_{T^m}\mu^{\otimes m}(\mathrm{d}\uu)\,\Bigl(
	     \one_\mathfrak{t}\(\shape[(\Tb,\cb)](\theta_{k,y_e}\uu)\)
	    -\one_\mathfrak{t}\(\shape[(\Tb,\cb)](\theta_{k,x}\uu)\) \Bigr),\\
\intertext{and}
\label{e:Bk}
	B_{k,j} &:= \eps^2 \sum_{(x,e) \in \lf(T)\times\edge(T)} \int_{T^m}\mu^{\otimes m}
		(\mathrm{d}\uu)\,\one_\mathfrak{t}\(\mathfrak{s}_{(\Tb,\cb)}(\theta_{k,y_e}\circ\theta_{j,x}\uu)\).
\end{align}
We use the notation $\mathfrak{t}_{\wedge k}\in \Clad[m-1]$ for the $(m-1)$-cladogram obtained from $\mathfrak{t}$ by deleting the leaf with
label $k$ (and relabelling the labels $j>k$ to $j-1$), i.e., if $\mathfrak{t}=\mathfrak{s}_{(\Tb,\cb)}(\uu)$, then $\mathfrak{t}_{\wedge k}=\mathfrak{s}_{(\Tb,\cb)}(\uu_{\wedge k})$ with $\uu_{\wedge k}:= (u_1,...,u_{k-1},u_{k+1}, ..., u_m)$.
Furthermore, for $\uu\in T^m$, we define
\begin{equation}
\label{e:Etk}
	E_{\mathfrak{t},k}(\uu) := \bset{v \in \Tb}{\mathfrak{s}_{(\Tb,\cb)}(\theta_{k,v}\uu)=\mathfrak{t}}.
\end{equation}
Note that $E_{\mathfrak{t},k}(\uu)$ does not depend on $u_k$, contains no $u_j$ for $j\ne k$, and that
$E_{\mathfrak{t},k}(\uu)\ne \emptyset$ only if $\shape(\uu_{\land
k}) = \mathfrak{t}_{\wedge k}$. In this case, $E_{\mathfrak{t},k}(\uu)$ ``corresponds to'' an edge of $\mathfrak{t}_{\wedge k}$. Let $\ell:=\delta\sum_{e\in\edge(T)}\delta_{y_e}$ be the uniform
distribution on $\set{y_e}{e\in\edge(T)}$.
By Fubini's theorem and using that $\eps\delta \sum_{(x,e)\in\lf\times\edge}\delta_{x} \otimes
\delta_{y_e}=\mu\otimes\ell$, we obtain
\begin{equation}
\label{e:A}
\begin{aligned}
	A_k &= \delta^{-1} \int_{\Tb^m}\mu^{\otimes m}(\mathrm{d}\uu)\(\ell(E_{\mathfrak{t},k}(\uu))-
\mu(E_{\mathfrak{t},k}(\uu))\)
\\
	    &= \int_{T^m}\mu^{\otimes m}(\mathrm{d}\uu)\one_{\tk}\((\mathfrak{s}_{(\Tb,\cb)}\uk)\) \cdot \(3\mu(E_{\mathfrak{t},k}(\uu)) + 1\)
\\
	    &= 3\Phi^{m,\mathfrak{t}}(\smallx) + \Phi^{m-1,\tk}(\smallx),
\end{aligned}
\end{equation}
where we have used in the second step that, because $(T,c)$ is binary,
\begin{equation}\label{e:018}
   \#\bset{e\in \edge(T)}{c(y_e,z,z')\in(z,z')}  =  2\#\bset{x\in \lf(T)}{c(x,z,z')\in(z,z')}+1
\end{equation}
for every $z,z'\in T$, and hence $\delta^{-1}\ell(\Etk)= 2N\mu(\Etk) + 1$ if $\shape[(\Tb,\cb)](\uk)=\tk$.
Similarly,
\begin{equation}
\begin{aligned}\label{e:B}
	B_{k,j}
 &=
    \tfrac\eps\delta\int_{\Tb}\ell(\mathrm{d}y)\int_{T}\mu(\mathrm{d}x)\,\int_{T^m}\mu^{\otimes m}(\mathrm{d}\uu)\,
    \one_{\mathfrak{t}}\(\shape[(\Tb,\cb)]\circ\theta_{k,y}\circ\theta_{j,x}(\uu)\)
  \\
 &=
    \tfrac\eps\delta \Phi^{m,\mathfrak{t}}(\smallx) + \eps A_k
    \\
    &= 2\Phi^{m,\mathfrak{t}}(\smallx) + \eps A_k+3\eps \Phi^{m,\mathfrak{t}}(\smallx).
\end{aligned}
\end{equation}
Combining \eqref{e:Omega0}, \eqref{e:A} and \eqref{e:B}, we obtain that
\begin{equation}\label{e:Omega1}
\begin{aligned}
	&\Omega_N \Phi^{m,\mathfrak{t}}(\smallx)
\\
 &=
   \sum_{k=1}^m \Phi^{m-1,\tk }(\smallx) + \(3m - 2m(m-1)\)\Phi^{m,\mathfrak{t}}(\smallx)-\eps(m-1)\sum_{k=1}^mA_k-3\eps m(m-1)\Phi^{m,\mathfrak{t}}(\smallx)
      \\
 &=
   \sum_{k=1}^m\Phi^{m-1,\tk }(\smallx)-m(2m-5)\Phi^{m,\mathfrak{t}}(\smallx)-\eps(m-1)\sum_{k=1}^m\Phi^{m-1,\tk }(\smallx)-6\eps m(m-1)\Phi^{m,\mathfrak{t}}(\smallx).
\end{aligned}
\end{equation}

For an edge $e$ of $\tk$, denote by $\tke$ the cladogram obtained by inserting a leaf labelled $k$ in $\tk$ at
the edge $e$ (and relabelling the
labels $j\ge k$ to $j+1$). In particular, $\tke$ is the cladogram obtained from $\testtree$ by
the Aldous move $(k, e)$. For $\uu\in T^m$, we have $\shape(\uk)=\tk$ if and only if there is an edge $e$ of $\tk$ such that
$\shape(\uu)=\tke$, and this $e$ is unique. Hence,
\begin{equation}
\label{e:020}
\begin{aligned}
    \sum_{k=1}^m\Phi^{m-1,\tk }(\smallx)
    &=
    \int_{T^m}\mu^{\otimes m}(\mathrm{d}\uu)\,\sum_{k=1}^m\one_{\tk} \(\shape(\uk)\)
    \\
    &= \int_{T^m}\mu^{\otimes m}(\mathrm{d}\uu)\,
    	\sum_{k=1}^m \sum_{e\in \edge(\tk)} \one_{\tke} \(\shape(\uu)\)
\\
    &= \int_{T^m}\mu^{\otimes m}(\mathrm{d}\uu)\, \Omega_m\one_\testtree\(\shape(\uu)\)
    	+ m\,\#\edge(\tk)\,\Phi^{m,\testtree}(\smallx)
.
\end{aligned}
\end{equation}
Inserting this into \eqref{e:Omega1} and using that $\#\edge(\tk) = 2m-5$, we see that

\begin{equation}
\label{e:Omega2}
	\bigl|\Omega_N\Phi^{m,\mathfrak{t}}(\smallx)-\int_{T^m}
		(\Omega_m\one_\mathfrak{t})\circ \shape \,\dx\mu^{\otimes m}\bigr|
    \le 7 m(m-1) \eps,
\end{equation}
which gives the claim.
\end{proof}

As $\Tbin$ is compact by Proposition~\ref{P:compact}, we can immediately conclude the following from the convergence of the generators.
\begin{corollary}[The limiting martingale problem]\label{C:tight}
Let\/ $(\smallx_N)_{N\in\N}$ be a sequence of random binary algebraic measure trees with\/
$\smallx_N\in\Tbin^N$, such that
\begin{equation}
   \smallx_N  \Rightarrow \smallx,\quad \text{as} \; N\rightarrow \infty,
\end{equation}
where\/ $\smallx$ is distributed according to\/ $P_0$ on\/  $\mathbb{T}_2^{\mathrm{cont}}$. Let\/
$X^N=(X_t^N)_{t\ge 0}$ be the Aldous chain started in\/ $\smallx_N$.
Then the sequence\/ $(X^N)_{N\in\N}$ is tight in Skorokhod path space.
Any limit process\/ $(X_t)_{t\geq 0}$ has continuous, $\Taldous$-valued paths, and satisfies the\/
$(\AldGen,{\mathcal D}(\AldGen),P_0)$-martingale problem.
\end{corollary}

\begin{proof}
\pstep{Tightness} Tightness follows, in view of the approximation result Proposition~\ref{p:001} and
Lemma~\ref{l:almostFeller}, with the
exactly same proof as Theorems~3.9.1 and 3.9.4 in \cite{EthierKurtz86} (see also \cite[Remark~4.5.2]{EthierKurtz86}).
\pstep{Continuous paths}
For $\Phi\in \CD(\AldGen)$, let $\Phi(X^N)=\(\Phi(X^N_t)\)_{t\ge 0}$.
By definition, $\CD(\AldGen)$ induces the topology of sample-shape convergence on $\Taldous$. Hence, continuity of
the paths of the limit process $X=(X_t)_{t\ge 0}$ in $\Taldous$ is equivalent to path-continuity of
$\Phi(X)$ for all $\Phi \in \CD(\AldGen)$. Because $\Phi(X)$ is the limit of $\Phi(X^N)$, this follows from the
obvious estimate
$\bigl|\Phi(X_t^N)-\Phi(X_{t-}^N)\bigr| \le \frac mN$ for $\Phi=\Phi^{m,\testtree}$.
\pstep{Values in $\Taldous$} That any limit point has $\Taldous$-valued paths follows directly from the fact that
	$X^N$ is $\Tbin^N$-valued, together with the approximation result Proposition~\ref{p:001}.
\pstep{Martingale problem}
That all limit points satisfy the martingale problem follows with the same proof as Lemma~4.5.1 in \cite{EthierKurtz86}.
\end{proof}

The following corollary is immediate from the previous corollary and the approximation result Proposition~\ref{p:001}.
\begin{corollary}[Existence]\label{C:exist}
For any probability measure\/ $P_0$ on\/  $\mathbb{T}_2^{\mathrm{cont}}$ there exists a solution in\/
$\C_{\mathbb{T}_2^{\mathrm{cont}}}(\R_+)$ of the\/   $(\AldGen,\CD(\AldGen),P_0)$-martingale problem.
\end{corollary}

\section{Duality, uniqueness and convergence}
\label{S:duality}
In this section we first obtain a  duality result that in turn  allows to conclude the uniqueness of the
martingale problem. We also use duality to show that the Aldous diffusion is a Feller process on $\Taldous$.
For $m\in\N$ let  $Y^m:=(Y^m_t)_{t\ge 0}$ be the  $\Clad$-valued Aldous chain
with generator $\Omega_m$ from \eqref{e:genN}. If $Y^m_0=\mathfrak{t}\in\Clad$, then $\mathbb{E}^Y_{\mathfrak{t}}$ denotes the corresponding expectation.

\begin{proposition}[Duality]\label{P:duality}
Let\/ $P_0$ be an arbitrary probability measure on\/  $\mathbb{T}_2^{\mathrm{cont}}$ and
let\/ $X:=((T_t,c_t,\mu_t))_{t\ge 0}$ be a solution of the\/  $(\AldGen,{\mathcal
D}(\AldGen),P_0)$-martingale problem  in\/ $\mathcal{D}_{\mathbb{T}_2^{\mathrm{cont}}}(\R_+)$.
For arbitrary\/ $m\in\N$ and\/ $\mathfrak{t}\in\Clad$, let\/ $Y^m:=(Y^m_t)_{t\ge 0}$ be the\/
$\Clad$-valued Aldous chain with\/ $Y^m_0=\mathfrak{t}$. Assume that\/  $Y^m$ is independent of\/ $X$. Then
\begin{equation}
\label{e:021}
\begin{aligned}
   \MoveEqLeft \mathbb{E}^X_{P_0}\big[\mu^{\otimes m}_t\big\{\uu\in T^m_t:\,\shape[(T_t,c_t)](\uu)=\mathfrak{t}\big\}\big]
   \\
   &=
   \int_{\mathbb{T}_2^{\mathrm{cont}}} \mathbb{E}^Y_{\mathfrak{t}}\big[\mu^{\otimes m}\big\{\uu\in
   T^m:\,\shape(\uu)=Y^m_t\big\}\big]\,P_0(\mathrm{d}(T,c,\mu)).
   \end{aligned}
\end{equation}
\end{proposition}

\begin{proof} Let $m\in\N$. For $\smallx=(T,c,\mu)\in\Tbin^{\mathrm{cont}}$ and $\mathfrak{t}\in\Clad$,
we define $H(\smallx,\mathfrak{t}):=\mu^{\otimes m}\bigl\{\uu\in T^m:\,\shape(\uu)=\mathfrak{t}\bigr\}$.
Then
\begin{equation}\label{e:022}
  \AldGen H(\cdot,\mathfrak{t})(\smallx)=\int_{T^m}\mu^{\otimes m}(\mathrm{d}\uu)\,\Omega_m\one_\mathfrak{t}\(\shape(\uu)\)
  =\Omega_mH(\smallx,\cdot)(\mathfrak{t}).
\end{equation}
By our assumptions on the test functions $H$ and definitions of $\AldGen$ and $\Omega_m$, the result follows by  \cite[Lemma~4.4.11,  Corollary~4.4.13]{EthierKurtz86}.
\end{proof}

\begin{corollary}[Uniqueness of the martingale problem]\label{C:unique}
Let\/ $P_0$ be an arbitrary probability measure on\/ $\mathbb{T}_2^{\mathrm{cont}}$. Then uniqueness holds for
the\/ $(\AldGen,{\mathcal D}(\AldGen),P_0)$-martingale problem in\/
$\mathcal{D}_{\mathbb{T}_2^{\mathrm{cont}}}(\R_+)$.
\end{corollary}

\begin{proof}
As the set of all shape polynomials is separating (for probability measures), the result is immediate by the previous proposition and  Proposition~4.4.7
from~\cite{EthierKurtz86}.
\end{proof}

\begin{corollary}[Feller process]\label{c:Feller}
	For\/ $F\in \C(\Taldous)$, $t\ge0$, and\/ $\smallx \in \Taldous$, let
	\begin{equation}
		S_tF(\smallx) := \Exp_\smallx\(F(X_t)\),
	\end{equation}
	where, under\/ $\Exp_\smallx$, $X=(X_t)_{t\ge 0}$ is the Aldous diffusion on\/ $\Taldous$ started in\/
	$\smallx$. Then\/ $(S_t)_{t\ge 0}$ is a Feller semi-group.
	In particular, the Aldous diffusion is a strong Markov process.
\end{corollary}
\begin{proof}
	$(S_t)_{t\ge 0}$ is well-defined by existence and uniqueness shown in Corollaries~\ref{C:exist} and
	\ref{C:unique}. It is a semi-group on the set of bounded measurable functions on $\Taldous$ by the Markov
	property of $X$, which in turn follows from uniqueness and Theorem~4.4.2(a) in \cite{EthierKurtz86}.
	Recall from Proposition~\ref{P:compact} and Lemma~\ref{l:polalgebra} that the state space $\Taldous$ is
	compact, and the set $\CD(\AldGen)$ of shape polynomials is uniformly dense in $\C(\Taldous)$. Hence, in order to
	show that $S_t$ maps $\C(\Taldous)$ into itself, it is enough to show $S_tF\in\C(\Taldous)$ for
	$F=\Phi^{m,\testtree}\in \CD(\AldGen)$. Using duality, we have $S_t\Phi^{m,\testtree}(\smallx) =
	\Exp\bigl[ \Phi^{m,Y_t}(\smallx) \bigr]$ for the $\Clad$-valued Aldous chain started in $\testtree$.
	Thus $S_t\Phi^{m,\testtree} \in \C(\Taldous)$, because it is a finite linear combination of the continuous
	functions $\Phi^{m,\testtree'}$ for different $\testtree'\in\Clad$.

	We have shown that $(S_t)_{t\ge 0}$ is a contraction semigroup on $\C(\Taldous)$.
	Its weak continuity follows directly from continuity of the sample paths of $X$.
	Weak continuity implies strong continuity, e.g., by Theorem~19.6 of \cite{Kallenberg2002}. Therefore,
	$X$ is a Feller process. This also implies the strong Markov property (e.g.\
	\cite[Theorem~4.2.7]{EthierKurtz86}).
\end{proof}

\section{Long term behavior and the Brownian CRT}
\label{S:longterm}
In this section, we define the algebraic measure Brownian CRT, and provide the joint density of the cladogram shape
spanned by a sample of finite size together with the vector of subtree masses branching off the edges of the
cladogram. Moreover, we show that the algebraic measure Brownian CRT is invariant under the Aldous diffusion and that
for any initial $\smallx\in\mathbb{T}_2^N$, the Aldous diffusion converges in law to the algebraic measure
Brownian CRT as time goes to infinity.

Recall the definition of the set $\Clad$ of $m$-cladograms (after Definition~\ref{Def:mclado}) and the shape
$\shape(u_1,...,u_m)$ spanned by the vector of $m$ points $u_1,...,u_m\in T$ (Definition~\ref{Def:shapeclado}).
We define
\begin{definition}[Algebraic measure Brownian CRT] \label{Def:amCRT}
The algebraic measure Brownian CRT is the unique (in distribution) random binary algebraic measure tree
$\X_{\mathrm{CRT}}=(T,c,\mu)$ with uniform annealed sample shape distribution, i.e., for all $m\in\N$,
for all $\mathfrak{t}\in\Clad$,
\begin{equation}
\label{e:013}
   \mathbb{E}_{\mathrm{CRT}}\Big[\mu^{\otimes m}\big\{(u_1,...,u_m):\,\shape\(u_1,.,,,.u_m\)=\mathfrak{t}\big\}\Big]=\tfrac{1}{\#\Clad}.
\end{equation}
\end{definition}

Note that there is a unique law on $\Tbin$ satisfying \eqref{e:013} because the sample shape distribution
separates probability measures on $\mathbb{T}_2$, and it is realized through the well-known Brownian CRT once we
ignore the distances (compare, \cite[Theorem~23]{Aldous1993}).

Now we provide the analog of \cite[Theorem~23]{Aldous1993} by considering, together with the sample shape, the
vector of masses of the subtrees branching off the edges of the shape cladogram. As expected, under the annealed
law of the Brownian CRT, we obtain that this vector is Dirichlet distributed and independent of the shape.
To state the result more precisely, for
$\uu=(u_1,\ldots,u_m) \in T$ let $T(\uu) = c(\{u_1,\ldots,u_m\}^3)$ be the generated subtree, and for
$e=\{e_x,e_y\}\in\edge(T(\uu))$ let
\begin{equation}
\eta_{(T,c,\mu)}(\uu, e) := \mu\Bigl(
	\Bset{v\in T}{c(v,e_x,e_y) \,\in\, (e_x,e_y) \cup \(\{e_x,e_y\} \cap \lf(T(\uu))\) }\Bigr).
\end{equation}
Let $\eta_{(T,c,\mu)}(\uu) = \(\eta_{(T,c,\mu)}(\uu,e)\)_{e\in \edge(T(\uu))}$ be the vector of these $2m-3$  masses
(assuming $u_1,\ldots,u_m$ are distinct).
We obtain the following, which is proven in the special case $m=3$ in \cite[Theorem~2]{Aldous1994}.
\begin{proposition}[Brownian CRT and $\Dir(\tfrac{1}{2},...,\tfrac{1}{2})$] \label{P:002}
Let\/ $\X_\mathrm{CRT}$ be the Brownian CRT, $m\in\N$, $\testtree \in \Clad$ and\/ $f\colon\Delta_{2m-3} \to \R$ bounded
measurable, where\/ $\Delta_k$ is the\/ $k$-simplex for\/ $k\in\N$, i.e.,
\begin{equation}
\label{e:Deltak}
   \Delta_{k}:=\big\{x\in [0,1]^k:\,x_1+...+x_{k} =1\big\}.
\end{equation}
Then the following holds:
\begin{equation}
\label{e:015}
\begin{aligned}
\MoveEqLeft{\mathbb{E}_{\mathrm{CRT}}\Bigl[\int\mu^{\otimes m}(\mathrm{d}\uu)
	\one_{\mathfrak{t}}\(\shape(\uu))\)f\(\eta_{(T,c,\mu)}(\uu)\)\Bigr]}\\
   &= \tfrac1{\#\Clad} \int_{\Delta_{2m-3}} f(\underline{x}) \;\Dirhalf(\mathrm{d}\underline{x}) \\
   &=\tfrac{\Gamma(m-\frac{3}{2})}{\#\Clad\Gamma(\frac{1}{2})^{2m-3}}
  	 \int_{\Delta_{2m-3}}f(\underline{x})\(x_1\cdot \ldots\cdot x_{2m-3}\)^{-\frac12}\,\mathrm{d}\underline{x},
   \end{aligned}
\end{equation}
where\/ $\Dirhalf$ is the Dirichlet distribution and the last integral in \eqref{e:015} is w.r.t.\ Lebesgue
measure on $\Delta_{2m-3}$.
\end{proposition}

\begin{proof} We follow Aldous's proof of \cite[Theorem~2]{Aldous1994}
and show a local limit theorem for the subtree mass distribution of a uniform $N$-cladogram (with uniform
distribution on the leaves) as $N\to\infty$. Because the uniform $N$-cladogram converges to the Brownian CRT,
this implies the claim.

Fix $m\in\N$ and $\testtree\in \Clad$. Let $N\ge m$.
Denote by $\pi_{N,m}\colon \Clad[N]\to\Clad$ the projection map which sends an $N$-cladogram $(T,c,\zeta)$
to the $m$-cladogram spanned by the first $m$ leaves, i.e., for $T_m:=c\(\zeta(\{1,...,m\})^3\)$,
\begin{equation}
\label{e:017}
   \pi_{N,m}(T,c,\zeta) :=\bigl(T_m,\, c\restricted{T_m},\, \zeta\restricted{\{1,...,m\}}\bigr).
\end{equation}
For $n=(n_e)_{e\in\edge(\mathfrak{t})}$ with $\sum_{e} n_e= N$, let $q_N(n)$ be the probability that
the first $m$ leaves of a uniform $N$-cladogram span the $m$-cladogram $\testtree$, and the numbers of leaves of
the $N$-cladogram in the subtrees corresponding to the edges of $\testtree$ are given by the vector $n$,
i.e.,
\begin{equation}
\label{e:016}
\begin{aligned}
   q_N(n) &:= \frac1{\#\Clad[N]} \#\Bset{(T,c,\zeta) \in \pi_{N,m}^{-1}(\testtree)}
			{N\cdot \eta_{(T,c,\zeta)}\(\zeta(1),\ldots,\zeta(m)\) = n} \\
   &= \frac1{\#\Clad[N]} \cdot
   	\frac{(N-m)!\prod_{e\in\extedge(\mathfrak{t})}\#\Clad[n_e+1] \prod_{e\in\intedge(\mathfrak{t})}\#\Clad[n_e+2]}
		{\prod_{e\in\extedge(\mathfrak{t})}(n_e-1)! \prod_{e\in\intedge(\mathfrak{t})}n_e!}.
\end{aligned}
\end{equation}
The first factor in the numerator together with the denominator counts the number of possibilities to distribute
the $N-m$ remaining leaves to the edges of the $m$-cladogram $\testtree$ (with quantities specified by $n$), and
the products in the numerator count the possibilities to give cladogram structure to the leaves associated to
every edge of $\testtree$. For an external edge $e\in\extedge(\mathfrak{t})$, this is the number of
$(n_e+1)$-cladograms (identify the additional leaf with the branch point of $\testtree$ it is attached to).
For an internal edge, we need two additional leaves. Recall that
\begin{equation} \label{e:019}
   \#\Clad[N] = (2N-5)!! = \tfrac{(2N-4)!}{2^{N-2}(N-2)!} \approx (N-2)! \cdot 2^{(N-2)}(\pi (N-2))^{-\frac12},
\end{equation}
where $\approx$ means that the quotient tends to $1$ as $N\to\infty$, and we have applied the
Stirling formula.

Fix $\eta=(\eta_1,...,\eta_{2m-3})\in\Delta_{2m-3}$ with $\eta_i>0$ for $i=1,\ldots,2m-3$. For $n_i=n_i(N)$ with
$\sum_{i=1}^{2m-3} n_i = N$ and $N^{-1}n_i \to \eta_i$ (in particular $n_i\to\infty$), as $N\to\infty$, we
obtain (using the convention that the first $m$ edges are external)
\begin{equation}
\label{qNm}
\begin{aligned}
q_N(n) &= \frac{1}{\#\Clad[N]} \cdot (N-m)!
		\cdot \prod_{i=1}^m \frac{\#\Clad[n_i+1]}{(n_i-1)!} \prod_{i=m+1}^{2m-3} \frac{\#\Clad[n_i+2]}{n_i!}
			\\
	&\approx
	 \frac{\sqrt{\pi(N-2)}}{(N-2)! \cdot 2^{N-2}} \cdot (N-m)!\prod_{i=1}^m2^{n_i-1}(\pi (n_i-1))^{-\frac{1}{2}}\prod_{i=m+1}^{2m-3}2^{n_i}(\pi n_i)^{-\frac{1}{2}}
	  \\
       &=
	 \sqrt{(N-2)} \cdot \tfrac{(N-m)!}{(N-2)!}\cdot 2^{2-m}\pi^{\frac12-\frac{2m-3}2}\prod_{i=1}^{m}(n_i-1)^{-\frac{1}{2}}
	 \prod_{i=m+1}^{2m-3}n_i^{-\frac12}
	 \\
	&\approx
		N^{-(m-\frac52)} \cdot (2\pi)^{2-m} \cdot N^{-(m-\frac{3}{2})} \cdot
	 \(\eta_1\eta_2\cdot ...\cdot\eta_{2m-3}\)^{-\frac12}
	 \\
	&=
	 N^{-(2m-4)}\cdot \frac1{\#\Clad[m]} \cdot \frac{\Gamma(\tfrac{2m-3}{2})}
	 {\Gamma(\frac{1}{2})^{2m-3}}\(\eta_1\eta_2\cdot ...\cdot\eta_{2m-3}\)^{-\frac12}.
\end{aligned}
\end{equation}
        This gives the claimed Dirichlet density on the $(2m-3)$-simplex in the limit.
\end{proof}

\begin{proposition}[Convergence to the CRT]\label{P:Aldlong}
Let\/ $(X_t)_{t\ge 0}$ be the Aldous diffusion started in\/ $\smallx\in\mathbb{T}_2^{\mathrm{cont}}$, and\/
${\mathcal X}_{\mathrm{CRT}}$ the algebraic measure Brownian CRT.  Then
\begin{equation}
\label{e:Aldlong}
   X_t \,\Tto\, {\mathcal X}_{\mathrm{CRT}}.
\end{equation}
In particular, the algebraic measure Brownian CRT is the unique invariant distribution of the Aldous diffusion.
\end{proposition}
\begin{proof}
Fix $m\in\N$ and $\testtree \in\Clad$. Let $(Y_t)_{t\ge 0}$ be the Aldous chain on $m$-caldograms started in
$Y_0=\testtree$.
Then, for $\Phi^{m,\testtree} \in \CD(\AldGen)$ as in \eqref{s:009}, we have by duality
(Proposition~\ref{P:duality})
\begin{equation}
	\Exp\(\Phi^{m,\testtree}(X_t)\) = \sum_{\testtree'\in\Clad} \mathbb{P}\{Y_t=\testtree'\}
	\Phi^{m,\testtree'}(\smallx).
\end{equation}
Because the uniform distribution on $\Clad$ is the unique reversible distribution of the Aldous chain (see
\cite{MR1774749}), $\lim_{t\to\infty}\mathbb{P}\{Y_t=\testtree'\}=\frac1{\#\Clad}$ for every
$\testtree'\in\Clad$. Because $\sum_{\testtree'\in\Clad} \Phi^{m,\testtree'} = 1$ on $\Taldous$, this means
\begin{equation}
	\lim_{t\to\infty} \Exp\(\Phi^{m,\testtree}(X_t)\) = \frac1{\#\Clad}
		= \Exp\(\Phi^{m,\testtree}(\X_\mathrm{CRT})\) .
\end{equation}
Because $\CD(\AldGen)$ is convergence determining for probability measures on $\Taldous$, this proves
\eqref{e:Aldlong}.
Invariance of the law of $\X_\mathrm{CRT}$ follows from the convergence \eqref{e:Aldlong} together with the
Feller property (Corollary~\ref{c:Feller}).
\end{proof}

In summary, we have now proven the first two theorems.
\begin{proof}[Proof of Theorems~\ref{T:Aldous} and \ref{T:diffappr}]
	Well-posedness of the martingale problem is shown in Corollaries~\ref{C:exist} and \ref{C:unique}.
	Continuous paths and tightness of the sequence of Aldous chains are shown in Corollary~\ref{C:tight}.
	Furthermore, every limit process satisfies the martingale problem (Corollary~\ref{C:tight}) and this implies
	convergence because of the uniqueness shown in Corollary~\ref{C:unique}. The Feller property is shown in
	Corollary~\ref{c:Feller}, and unique ergodicity with the algebraic measure Brownian CRT as invariant
	distribution is shown in Proposition~\ref{P:Aldlong}.
\end{proof}

\section{On the dynamics of the sample subtree mass vector}
\label{S:subtreemass}
In this section, we further study the Aldous diffusion on binary, algebraic non-atomic measure trees and prove
Theorem~\ref{t:massgen}.
Recall from Proposition~\ref{P:002} that under the annealed law of the Brownian CRT, the sample tree shape is
uniform and independent of the vector of subtree masses branching of the cladogram spanned by the sample.
Furthermore, the vector of subtree masses is Dirichlet distributed.
Next, we study the infinitesimal evolution of the quenched law of this vector under the dynamics of the Aldous diffusion in
the case of sample size $m=3$.

Recall the definition of the components $\Sub_v(u)$, $u,v\in T$, from \eqref{e:equiv}, and from \eqref{e:eta3}
that $\ueta(\uu)$ with $\uu=(u_1,u_2,u_3) \in T^3$ denotes the vector of the three masses of the components connected to $c(\uu)$,
i.e.
\begin{equation}\label{e:etarepeat}
	\ueta(\uu) = \(\eta_i(\uu)\)_{i=1,2,3} = \( \mu\(\Sub_{c(\uu)}(u_i)\) \)_{i=1,2,3}.
\end{equation}
With a slight abuse of notation, we also denote for $v\in\br(T)$ by $\underline\eta(v)$ the $\mu$-masses of the three
components of $T\setminus \{v\}$ (ordered decreasingly for definiteness), so that $\ueta(\uu)=\ueta(c(\uu))$ up
to a permutation of the entries of the vector.
Also recall that for mass-polynomials
\begin{equation} \label{e:Phirepeat}
   \Phi^{f}(\smallx) = \int f\(\ueta(c(\uu))\)\, \mu^{\otimes 3}(\mathrm{d}\uu)
\end{equation}
with $f\in\C^2\([0,1]^3\)$ and $\smallx=(T,c,\mu)\in \Tbin$, we define
\begin{equation}
\label{e:generatorrepeat}
\begin{aligned}
   \AldGenMass \Phi^f(T,c,\mu)
	&= \int_{T^3}\dx\mu^{\otimes3}\,\Bigl( 2\sum_{i,j=1}^3 \eta_i(\delta_{ij} - \eta_j) \partial_{ij}^2 f(\underline\eta)
		+3\sum_{i=1}^3 (1-3\eta_i)\partial_i f(\underline\eta) \\
  &\phantom{{}+\int_{T^3}\dx\mu^{\otimes3}\,\Bigl({}} + \tfrac{1}{2}\sum_{i,j=1,\, i\ne j}^3 \Theta_{i,j}f(\ueta)
		+ \sum_{i=1}^3\(f(e_i)-f(\underline\eta)\)\Bigr).
\end{aligned}
\end{equation}
see \eqref{e:Phi} and \eqref{e:generator}, and the definition of migration operators $\Theta_{i,j}$ in
\eqref{e:migration}.
Note that $\Phi^f \in \C(\Tbin)$ by Proposition~\ref{P:conveq}, and
because the functions $\Theta_{ij}f$ are continuous due to continuous differentiability of $f$,
$\AldGenMass \Phi^f$ is also a mass-polynomial. In particular, $\AldGenMass\Phi^f\in \C(\Taldous)$.

\begin{enremark}\label{r:Aldsym}
\item	If $f\in\C^2([0,1]^3)$ is symmetric, we can use the symmetry of the sampling procedure and rewrite
	\eqref{e:generatorrepeat} as
\begin{equation}\label{e:Aldsym}\begin{aligned}
	\AldGenMass \Phi^f(\smallx) &= 3 \int_{T^3} \dx\mu^{\otimes3}\, \Bigl(
		2\eta_1(1-\eta_1) \partial_{11}^2 f(\ueta) - 4\eta_1\eta_2\partial_{12}^2 f(\ueta)
		+ 3(1-3\eta_1)\partial_1f(\ueta)\\
	&\phantom{{}=3\int} + \Theta_{1,2}f(\ueta) + f(1,0,0) - f(\ueta) \Bigr).
\end{aligned}\end{equation}
This helps to reduce the number of terms in explicit calculations.
\item We can often assume $f$ to be symmetric. If $f$ is not necessarily symmetric, we use that
$\Phi^f=\Phi^{\tilde{f}}$ for the symmetrization $\tilde{f}$ of $f$ defined as follows.
Given a permutation $\pi$ of
$\{1,2,3\}$, define $f_\pi(x_1,x_2,x_3) := f(x_{\pi(1)},x_{\pi(2)},x_{\pi(3)})$.
Then $\tilde{f} = \frac16 \sum_{\pi \in S_3} f_\pi$.
\end{enremark}

For the proof of Theorem~\ref{t:massgen}, we do not use the martingale problem of Theorem~\ref{T:Aldous}, because
approximating the mass polynomial of degree three by shape polynomials, though possible in theory, seems
difficult in praxis. Instead, we show that uniform convergence of generators holds also for mass polynomials of
degree three, and use the diffusion approximation of Theorem~\ref{T:diffappr}.
For $N\in\N$, recall the state space $\mathbb{T}_2^{N}$ from \eqref{e:T2N} and the Aldous chain with generator
$\Omega_N$ from \eqref{e:genN}.

\begin{proposition}[Subtree mass under the Aldous diffusion] \label{p:massgen}
For all test functions\/ $\Phi^f$ of the form \eqref{e:Phi} with\/
$f\colon[0,1]^3\to\R$ twice continuously differentiable,
\begin{equation}\label{e:result}
	\lim_{N\to\infty} \sup_{\smallx\in \Tbin^N}
		\bigl| \Omega_N\Phi^f(\smallx) - \AldGenMass\Phi^f(\smallx) \bigr|
		= 0.
\end{equation}
\end{proposition}

From here we can prove Theorem~\ref{t:massgen} by standard arguments.

\begin{proof}[Proof of Theorem~\ref{t:massgen}]
	Let $X=(X_t)_{t\ge 0}$ be the Aldous diffusion on $\Taldous$.
	Due to Proposition~\ref{p:001}, there exist $\Tbin^N$-valued random variables $X^N_0$ such that $X^N_0$
	converges in law to $X_0$. By Theorem~\ref{T:diffappr}, the Aldous chains $X^N$ started in $X^N_0$
	converge in law to $X$.  Furthermore, $\AldGenMass$ maps into $\C(\Tbin)$.
	Hence the same proof as Lemma~4.5.1 in \cite{EthierKurtz86} shows that \eqref{e:071} follows from
	Proposition~\ref{p:massgen}.
\end{proof}

\allowdisplaybreaks	Recall the intrinsic metric $r_\mu$ on an algebraic measure tree $(T,c,\mu)$, as defined in \eqref{e:metric}.
Before proving Proposition~\ref{p:massgen}, we show how to use the extended martingale problem from
Theorem~\ref{t:massgen} to calculate the annealed average $r_\mu$-distance in the algebraic measure Brownian CRT.

\begin{corollary}[Mean average distance of the Brownian CRT]\label{cor:001}
Let\/ $\X_\mathrm{CRT}=(T,c,\mu)$ be the algebraic measure Brownian CRT. Then
\begin{equation}
\label{e:070}
    \mathbb{E}_{\mathrm{CRT}}\Big[\int_{T^2}r_\mu(x,y)\,\mu^{\otimes 2}(\mathrm{d}(x,y))\Big]=\tfrac25.
\end{equation}
\end{corollary}
\begin{proof}
First, we express the average distance $\Phi(\smallx) := \int r_\mu(x,y)\,\mu^{\otimes 2}(\mathrm{d}(x,y))$ for
$\smallx=(T,c,\mu) \in \Taldous$ as mass polynomial of degree three.
Recall the branch point distribution $\nu=\mu^{\otimes 3} \circ c^{-1}$ used in the definition of $r_\mu$.
\begin{equation}
\label{e:disntance}
\begin{aligned}
   \Phi(\smallx) &= \int \sum_{v\in\br(T)\cap(x,y)}\nu(\{v\})\,\mu^{\otimes 2}(\mathrm{d}(x,y))
   \\
   &= \sum_{v\in\br(T)}\nu(\{v\})\mu^{\otimes 2}\(\{(x,y):\,v\in(x,y)\}\)
   \\
   &= 2\sum_{v\in\br(T)}\nu(\{v\})\(\eta_1(v)\eta_2(v)+\eta_2(v)\eta_3(v)+\eta_1(v)\eta_3(v)\)
   \\
   &= 2\int \(\eta_1(\uu)\eta_2(\uu) + \eta_2(\uu)\eta_3(\uu) + \eta_1(\uu)\eta_3(\uu)\)
   		\,\mu^{\otimes 3}(\mathrm{d}\uu),
\end{aligned}
\end{equation}
i.e.\ $\Phi=2\Phi^f$ (on $\Taldous$) with $f(x,y,z):=xy+yz+xz$. From here, we could obtain \eqref{e:070} by a
direct computation using Proposition~\ref{P:002}. But it also follows easily from the invariance of the CRT
under the Aldous diffusion:
Because $f$ is symmetric, we obtain from Remark~\ref{r:Aldsym}
\begin{equation}\label{e:Phif}
\begin{aligned}
	\AldGenMass\Phi^f(\smallx) &= 3 \int \(-4\eta_1\eta_2 + 3(1-3\eta_1)(\eta_2+\eta_3)
		+ \frac{(\eta_1+\eta_2)\eta_3 - f(\ueta)}{\eta_1} - 3\eta_1\eta_2\)\,\dx\mu^{\otimes3} \\
	&= 3\int ( 5\eta_1 - 25\eta_1\eta_2)\,\dx\mu^{\otimes 3} = 5 - 25 \Phi^f(\smallx).
\end{aligned}
\end{equation}
Thus $\Exp_{\mathrm{CRT}}[\AldGenMass\Phi^f]=0$ implies 
$\Exp_{\mathrm{CRT}}[\Phi^f(\smallx)]=\frac15$.
\end{proof}

To prove Proposition~\ref{p:massgen}, fix $N\in\N$ and $\smallx=(T,c,\mu)\in \mathbb{T}_2^{N}$. We use some
notation from the proof of Proposition~\ref{P:generator}, in particular recall from \eqref{e:epsN} that $\eps$
and $\delta$ denote the inverse numbers of leaves and edges, respectively, and the extended tree $(\Tb, \cb)$
which allows to represent one Aldous move on the same tree (see Figure~\ref{f:Tbar}).
We consider $\mu$, $\nu$, and $\underline{\eta}$ to be defined on $(\Tb,\cb)$ and, for $z\in
\lf(T)\times\edge(T)$, we denote by $\mu_z$, $\nu_z$, and $\uetaz$ the corresponding objects
after the Aldous move $z$.
Because our trees are binary, the relation between the fraction of leaves and the fraction of edges in a subtree
can be easily related as follows.

\begin{lemma}[Proportion of leaves versus edges] \label{l:lambda}
	Let\/ $\smallx=(T,c,\mu)\in\mathbb{T}_2^N$, and\/ $S=\Sub_v(u)$ for some\/
	$v\in \br(T)$, $u\in T\setminus\{v\}$ a component.
	Let\/ $\ell(S)$ be the fraction of the edges contained in\/ $S$ (including the edge to\/ $v$). Then
	\begin{equation}
		\ell(S) = \mu(S)\cdot (1+3\delta) - \delta.
	\end{equation}
\end{lemma}

Recall the branch point distribution $\nu=\mu^{\otimes 3} \circ c^{-1}$.
The next Lemma shows the effect we would see if the branch point distribution remained unchanged.
This corresponds exactly to Aldous's original calculation (compare with (\ref{e:FVnegative}) but notice that our
chain runs at total rate $N(2N-3)$ rather than $N^2$ as in Aldous's case).
In what follows, we write $O(\eps)$ for terms which divided by $\eps=\eps_N$ are bounded uniformly in the tree (and
$N$), while the bound may depend on $f$, and similarly for $o(\eps)$ and so on.

\begin{lemma}[Wright-Fisher term with negative drift]\label{l:deltaf}
Let\/ $f$ be as in Proposition~\ref{p:massgen}. Then
\begin{equation}
\label{e:deltaf}
\begin{aligned}
   \MoveEqLeft\sum_{z\in \lf(T)\times\edge(T)}\; \sum_{v\in \br(T)} \nu\{v\} \(f(\uetaz(v)) - f(\underline\eta(v))\)
   \\
	&=
   \int \Bigl(2\sum_{i,j=1}^3 \eta_i(\delta_{ij}-\eta_j)\partial_{ij}^2 f(\ueta)
		   - \sum_{i=1}^3 (1-3\eta_i)\partial_i f(\ueta)\Bigr)\,\mathrm{d}\mu^{\otimes 3}  + o(1),
\end{aligned}
\end{equation}
and the\/ $o(1)$-term tends to zero as\/ $N\to\infty$ uniformly in the binary tree with\/ $N$ leaves.
\end{lemma}

\begin{proof}
According to Remark~\ref{r:Aldsym}, we may and do assume w.l.o.g.\ that $f$ is symmetric so that
$f(\ueta(\uu))$ depends on $\uu\in T^3$ only through $v=c(\ueta(\uu))$.
Fix $v\in\br(T)$.  To make the calculation more readable, we abbreviate $\eta_i = \eta_i(v)$ (ordered
decreasingly) in the following equations as long as $v$ is fixed.
Denote the three components of $T\setminus\{v\}$ by $S_i(v)$, $i=1,2,3$, ordered such that
$\eta_i=\mu(S_i)$.
For all $z=(u,e)\in\lf(T)\times\edge(T)$ with $u\in S_i(v)$ and $e\in S_j(v)$, a Taylor
expansion yields
\begin{equation}
\label{e:fz}
   f\(\uetaz(v)\)
   =
   \(1-\eps(\partial_i-\partial_j)+\tfrac{1}{2}\eps^2(\partial_i-\partial_j)(\partial_i-\partial_j)\)f(\ueta) +o(\eps^2).
\end{equation}
Using first this expansion and then Lemma~\ref{l:lambda}, we obtain
\begin{equation}
\begin{aligned}
\label{e:Av}
	A_v
  &:=
     \sum_{z\in \lf(T)\times\edge(T)} \(f(\uetaz)-f(\ueta)\)
     \\
     &=
     \sum_{i,j=1,\, i\ne j}^3 \tfrac{\eta_i}{\eps} \tfrac{\ell_j}{\delta}
   \eps\Bigl( \( \partial_j - \partial_i + \tfrac\eps2 (\partial_{ii} + \partial_{jj} - 2\partial_{ij})
   \)f(\ueta)+o(\eps) \Bigr)
     \\
	&=
   \sum_{i,j=1,\, i\ne j}^3 \eta_i \frac{\eta_j(1+3\delta) -\delta}{\delta}
   \Bigl( \( \partial_j - \partial_i + \tfrac\eps2 (\partial_{ii} + \partial_{jj} - 2\partial_{ij})
   \)f(\ueta)+o(\eps) \Bigr).
\end{aligned}
\end{equation}
As the highest order term is anti-symmetric in $i\not=j$, i.e.\ $\sum_{i\not=j=1}^3
\eta_i\eta_j(\partial_j - \partial_i)f(\underline{\eta}) = 0$, and $\frac\eps\delta=2+O(\eps)$, we obtain
\begin{equation}\label{e:}
\begin{aligned}
	A_v
  &=
    -\!\sum_{i,j=1,\,i\ne j}^3 \eta_i(\partial_j - \partial_i)f(\ueta)
	+\tfrac\eps\delta\sum_{i,j=1,\, i\ne j}^3\eta_i\eta_j\(\partial_{ii}-\partial_{ij}\)f(\ueta)
  +o(1)
    \\
   &=-\sum_{i=1}^3\(1-3\eta_i\)\partial_{i}f(\ueta)
   +2\sum_{i=1}^3\eta_i\(1-\eta_i\)\partial_{ii}f(\ueta)-2\sum_{i,j=1,\, i\ne j}^3\eta_i\eta_j\partial_{ij}f(\ueta) + o(1)
       \\
   &=
     -\sum_{i=1}^3(1-3\eta_i)\partial_{i}f(\underline{\eta})
	 + 2\sum_{i,j=1}^3 \eta_i(\delta_{ij}-\eta_j)\partial_{ij}f(\underline{\eta}) + o(1),
\end{aligned}
\end{equation}
and the claim follows by (weighted) summation over $v$ and Fubini's Theorem.
\end{proof}

Recall that, for $e\in\edge(T)$, we introduced $x_e,y_e\in \Tb\setminus T$, where $x_e$ is ``in the middle'' of
$e$ and $y_e$ is a leaf attached to $x_e$. Because $\mu$ is supported by $T$ and $\eta(x_e)$ is ordered
decreasingly, we always have $\eta_3(x_e)=0$ and $\eta_1(x_e)+\eta_2(x_2)=1$.
The following lemma is easily obtained by associating a branch point to its three adjacent edges.
\begin{lemma}[Matching lemma]\label{l:edgematch}
Let\/ $g\colon [0,1]^2\to \R$ be symmetric. Then
\begin{equation}\label{e:edgematch}
	\sum_{e\in\edge(T)} g(\eta_1(x_e),\eta_2(x_e))
		= \tfrac12 \sum_{v\in \br(T)} \sum_{i=1}^3 g\(\eta_i(v), 1-\eta_i(v)\) + \tfrac12Ng(1-\eps,\eps)
\end{equation}
\end{lemma}
\begin{proof}
	If $e\in \edge(T)$ is adjacent to $v\in \br(T)$, there is $i\in\{1,2,3\}$, $j\in\{1,2\}$ with
	$\eta_j(x_e) = \eta_i(v)$ and $\eta_{3-j}(x_e)=1-\eta_i(v)$.
	The edge $e$ is either adjacent to precisely two branch points, or it is an external edge, i.e.\
	adjacent to a leaf of $T$. In the latter case, we have
	$\(\eta_1(x_e), \eta_2(x_e)\)=(1-\eps,\eps)$, and there are $N$ external edges. Therefore,
	\begin{equation}
	2\!\sum_{e\in\edge(T)} g(\eta_1(x_e),\eta_2(x_e)) - Ng(1-\eps,\eps)
		= \sum_{v\in \br(T)} \sum_{i=1}^3 g\(\eta_i(v), 1-\eta_i(v)\)
	\end{equation}
	as claimed.
\end{proof}

\begin{proof}[Proof of Proposition~\ref{p:massgen}]
We assume w.l.o.g.\ that $f$ is symmetric (see Remark~\ref{r:Aldsym}).

\pstep{Step~1} Recall that for $z=(u,e)\in \lf(T)\times\edge(T)$, $\nu_z$ and $\ueta^z$ are the branch point
distribution and mass vector after the Aldous move $z$, respectively.
In the first step, we calculate the effect of the branch point ``created'' by the Aldous move $z$ due to the fact
that $\nu_z(\{x_e\})$ might be non-zero, whereas $\nu(\{x_e\})=0$ for all $e\in\edge(T)$. To this end, set
\begin{equation}\label{e:Cv}
	C_{x_e} := \sum_{z\in\lf(T)\times\edge(T)}\nu_z\{x_e\}f\(\uetaz(x_e)\).
\end{equation}
Recall that we order the entries of $\ueta$ decreasingly, so that $\eta_1(x_e)+\eta_2(x_e)=1$ and
$\eta_3(x_e)=0$.
For $(x,y,z)\in \Delta_3$, let
\begin{equation}
\begin{gathered}
	h_\eps(x,y,z):=\(1-\eps(2 + x\partial_1 + y\partial_2 - \partial_3)\)f(x,y,z),\\
	g_\eps(x,y):=6xy\cdot h_\eps(x,y,0).
\end{gathered}
\end{equation}
Then $g_\eps$ is a symmetric function.
Let $e\in\edge(T)$. For $z=(u,e')\in \lf(T)\times\edge(T)$, we have $\nu_z\{x_e\}\ne 0$ if and only if $e = e'$, and
hence, using symmetry of~$f$,
\begin{equation}\label{e:nuxe}
\begin{aligned}
	C_{x_e} &= \sum_{u\in\lf(T)} \nu_{(u,e)}\{x_e\}f\(\ueta^{(u,e)}(x_e)\)\\
	&=\sum_{i=1}^2 N\eta_i(x_e) \cdot 6\eps(\eta_i(x_e)-\eps)\eta_{3-i}(x_e)
		\cdot f\(\eta_i(x_e)-\eps,\, \eta_{3-i}(x_e),\, \eps\) \\
	&=  g_\eps\(\eta_1(x_e),\, \eta_2(x_e)\) + O(\eps^2),
\end{aligned}
\end{equation}
where we used, in the last equality, a first order Taylor expansion of $f$ and the identity $\eta_1(x_e) +
\eta_2(x_e)=1$.
For $v\in \br(\Tb)\setminus T$, there is a unique edge $e\in\edge(T)$ with $v=x_e$. Thus, using Lemma~\ref{l:edgematch},
\begin{equation}\label{e:newbr}
\begin{aligned}
	\sum_{v\in \br(\Tb)\setminus T} C_v 
	&= \sum_{e\in\edge(T)} g_\eps\(\eta_1(x_e),\, \eta_2(x_e)\) + O(\eps) \\
	&= \tfrac12\sum_{v\in \br(T)} \sum_{i=1}^3 g_\eps\(\eta_i(v),\, 1-\eta_i(v)\) +
		\tfrac12 N g_\eps(1-\eps,\eps) + O(\eps)
\end{aligned}
\end{equation}
Now we use that $\nu\{v\}=6\eta_1(v)\eta_2(v)\eta_3(v)$ for $v\in \br(T)$,
and hence for any permutation $(i,j,k) \in S_3$, we have
$g_\eps\(\eta_i(v),\, 1-\eta_i(v)\)=\nu\{v\}(\frac1{\eta_j(v)} +
\frac1{\eta_k(v)})h_\eps(\eta_i,1-\eta_i,0)$,
and obtain
\begin{equation}
	\sum_{v\in \br(\Tb)\setminus T} C_v
		= \tfrac12\sum_{v\in \br(T)} \nu\{v\} \sum_{i,j=1,\,i\ne j}^3 \frac{h_\eps\(\eta_i(v), 1-\eta_i(v), 0\)}{\eta_j(v)}
		+ 3f(1,0,0) + O(\eps).
\end{equation}

\pstep{Step~2}
In the second step, we calculate the effect of the change in $\nu\{v\}$ for the ``old'' branch points.
Fix $v\in\br(T)$. To make the calculation more readable, we abbreviate $\eta_i = \eta_i(v)$ as long as $v$ is
fixed.
We use Lemma~\ref{l:lambda} in the first transformation, and a first order Taylor expansion of $f$ in the second.
\begin{equation}\label{e:bvcalc}
\begin{aligned}
	B_v &:= \sum_{z\in \lf(T)\times\edge(T)} \(\nu_z\{v\} - \nu\{v\}\) f\(\uetaz(v)\) \\
	&= \sum_{(i,j,k)\in S_3} \frac{\eta_i}{\eps} \frac{\eta_j(1+3\delta) - \delta}{\delta}\cdot
		\nu\{v\}\Bigl(\frac{(\eta_i-\eps)(\eta_j+\eps)\eta_k}{\eta_i\eta_j\eta_k} - 1\Bigr)
		f(\eta_i-\eps,\eta_j+\eps, \eta_k)\\
	&= \nu\{v\}\!\sum_{(i,j,k)\in S_3}  \(\tfrac1\delta{\eta_j(1+3\delta)} - 1\)
		\frac{\eta_i-\eta_j-\eps}{\eta_j}
		\(f(\ueta) + \eps(\partial_j - \partial_i)f(\ueta) + O(\eps^2)\)\\
\end{aligned}
\end{equation}
Cancelling all terms which are anti-symmetric in $(i,j)$, we obtain
\begin{equation}\label{e:bvcont}
\begin{aligned}
	B_v &= \nu\{v\}\!\sum_{(i,j,k)\in S_3} \Bigl(
		\tfrac\eps\delta (\eta_i-\eta_j)(\partial_j - \partial_i)f(\ueta)
		- \( \frac{\eta_i-\eps}{\eta_j} - 1 + \tfrac\eps\delta \)f(\ueta)\\
	&\phantom{=\nu\{v\}\!\sum_{(i,j,k)\in S_3} \Bigl(}\, - \eps\frac{\eta_i}{\eta_j}(\partial_j - \partial_i)f(\ueta)
	+ O(\eps) \Bigr).
\end{aligned}
\end{equation}
Using $\eps/\delta=2+O(\eps)$, that $\sum_{(i,j,k)\in S_3} \frac{\eta_i}{\eta_j} =
\sum_{(i,j,k)\in S_3}(\tfrac1{2\eta_j} -\frac12)$, and
\begin{equation}\label{e:partre}
	\sum_{(i,j,k)\in S_3} \tfrac{\eta_i}{\eta_j} (\partial_j - \partial_i)f
		= \sum_{j=1}^3 \( \tfrac{1-\eta_j}{\eta_j}\partial_j - \tfrac{\eta_i}{\eta_j}\partial_i -
			\tfrac{\eta_k}{\eta_j}\partial_k \)f
		= \sum_{(i,j,k)\in S_3}  \tfrac1{2\eta_j}(\partial_j - \eta_i\partial_i - \eta_k
		\partial_k)f + O(1),
\end{equation}
we continue
\begin{equation}\label{e:Bv}
\begin{aligned}
B_v &= \nu\{v\}\!\sum_{(i,j,k)\in S_3} \Bigl( 4(\eta_j-\eta_i)\partial_if(\ueta)
	  - \( \tfrac1{2\eta_j} + \tfrac12  - \tfrac\eps{\eta_j}
		+ \tfrac\eps{2\eta_j} ( \partial_j - \eta_i\partial_i -
		\eta_k\partial_k) \)f(\ueta)
	  + O(\eps) \Bigr)\\
	&= 4\nu\{v\}\! \sum_{i=1}^3 (1-3\eta_i)\partial_if(\ueta)
	    - \nu\{v\}\! \sum_{(i,j,k)\in S_3} \frac{h_\eps(\eta_i, \eta_j, \eta_k)}{2\eta_k}
	    - \nu\{v\}\(3f(\ueta) + O(\eps)\).
\end{aligned}
\end{equation}

\pstep{Step 3}
Because $f$ is twice continuously differentiable, $\frac1z(2+x\partial_1+y\partial_2 -
\partial_3)\(f(x,y+z,0)-f(x,y,z)\)$ is bounded and hence
\begin{equation}
\begin{aligned}
	\tfrac1{z}\(h_\eps(x,y+z,0) - h_\eps(x,y,z)\) &= \tfrac1{z} \(f(x,y+z,0) - f(x,y,z)\) + O(\eps)\\
	&= \Theta_{3,2}f(x,y,z) + O(\eps).
\end{aligned}
\end{equation}
Therefore, using Fubini's Theorem and combining \eqref{e:newbr} with \eqref{e:Bv} yields
\begin{equation}\label{e:deltanu}
\begin{aligned}
\MoveEqLeft{\sum_{z\in \lf(T)\times\edge(T)}\; \sum_{v\in \br(\Tb)}  \(\nu_z\{v\} - \nu\{v\}\) f\(\uetaz(v)\)}\\
	&= \sum_{v\in \br(T)} B_v + \sum_{v\in \br(\Tb)\setminus T} C_v\\
	&= \sum_{v\in \br(T)}\!\! \nu\{v\} \Bigl( 4\sum_{i=1}^3 (1-3\eta_i)\partial_if\(\ueta(v)\)
	    + \tfrac12\!\!\!\sum_{i,j = 1,\,i\ne j}^3\!\!\! \Theta_{i,j}f\(\ueta(v)\)
	 + 3f(1,0,0) - 3f\(\ueta(v)\) \Bigr) \\
	 &\phantom{{}={}} + O(\eps).
	\end{aligned}
\end{equation}
Together with Lemma~\ref{l:deltaf} (and using symmetry of $f$), we have obtained for $\smallx\in \Tbin^N$
\begin{equation}\label{e:42}
\begin{aligned}
	\Omega_N\Phi^f(\smallx) &=
		\sum_{z\in\lf(T)\times\edge(T)}\; \sum_{v\in\br(\Tb)}
				\Bigl(\nu_z\{v\}f\(\ueta^z(v)\) - \nu\{v\}f\(\ueta(v)\)\Bigr) \\
	  &= \AldGenMass\Phi^f(\smallx) + o(1),
\end{aligned}
\end{equation}
which shows the claim of Proposition~\ref{p:massgen}.
\end{proof}

\bibliographystyle{alpha}
\bibliography{gpw}

\begin{thebibliography}{BRST02}

\bibitem[Ald93]{Aldous1993}
David Aldous.
\newblock The continuum random tree {III}.
\newblock {\em Ann. Probab.}, 21:248--289, 1993.

\bibitem[Ald94]{Aldous1994}
David Aldous.
\newblock Recursive self-similarity for random trees, random triangulations and
  {B}rownian excursion.
\newblock {\em Ann. Probab.}, 22(2):527--545, 1994.

\bibitem[Ald00]{MR1774749}
David Aldous.
\newblock Mixing time for a {M}arkov chain on cladograms.
\newblock {\em Combin. Probab. Comput.}, 9(3):191--204, 2000.

\bibitem[ALW17]{AthreyaLohrWinter2017}
Siva Athreya, Wolfgang L{\"o}hr, and Anita Winter.
\newblock Invariance principle for variable speed random walks on trees.
\newblock {\em Ann.\ Probab.}, 45(2):625--667, 2017.

\bibitem[AS01]{MR1841949}
Benjamin~L. Allen and Mike Steel.
\newblock Subtree transfer operations and their induced metrics on evolutionary
  trees.
\newblock {\em Ann. Comb.}, 5(1):1--15, 2001.

\bibitem[BHV01]{MR1867931}
Louis~J. Billera, Susan~P. Holmes, and Karen Vogtmann.
\newblock Geometry of the space of phylogenetic trees.
\newblock {\em Adv. in Appl. Math.}, 27(4):733--767, 2001.

\bibitem[BRST02]{MR1920237}
Oliver Bastert, Dan Rockmore, Peter~F. Stadler, and Gottfried Tinhofer.
\newblock Landscapes on spaces of trees.
\newblock {\em Appl. Math. Comput.}, 131(2-3):439--459, 2002.

\bibitem[DGP12]{DepperschmidtGrevenPfaffelhuber2012}
Andrej Depperschmidt, Andreas Greven, and Peter Pfaffelhuber.
\newblock Tree-valued {F}leming-{V}iot dynamics with selection.
\newblock {\em Ann.\ Appl.\ Probab.}, 22(6):2560--2615, 2012.

\bibitem[EGW17]{EvansGruebelWakolbinger2017}
Steven~N. Evans, Rudolf Gr\"ubel, and Anton Wakolbinger.
\newblock Doob-{M}artin boundary of {R\'e}my's tree growth chain.
\newblock {\em Ann.\ Probab.}, 45(1):225--277, 2017.

\bibitem[EK86]{EthierKurtz86}
Stewart~N. Ethier and Thomas~G. Kurtz.
\newblock {\em Markov {P}rocesses. {C}haracterization and {C}onvergence}.
\newblock John Wiley, New York, 1986.

\bibitem[EW06]{EvansWinter2006}
Steven~N. Evans and Anita Winter.
\newblock Subtree prune and re-graft: A reversible real-tree valued {M}arkov
  chain.
\newblock {\em Ann.\ Probab.}, 34(3):918--961, 2006.

\bibitem[Fel03]{Felsenstein2003}
Joseph Felsenstein.
\newblock {\em Inferring Phylogenies}.
\newblock Sinauer, 2003.

\bibitem[For]{Forman2018}
Noah Forman.
\newblock Mass structure of weighted real trees.
\newblock arXiv:1801.02700v1.

\bibitem[FPRWa]{FormanPalRizzoloWinkel2018c}
Noah Forman, Soumik Pal, Douglas Rizzolo, and Matthias Winkel.
\newblock Aldous diffusion {I}: a projective system of continuum $k$-tree
  evolutions.
\newblock arXiv:1809.07756.

\bibitem[FPRWb]{FormanPalRizzoloWinkel2016}
Noah Forman, Soumik Pal, Douglas Rizzolo, and Matthias Winkel.
\newblock Diffusions on a space of interval partitions with
  {P}oisson-{D}irichlet stationary distributions.
\newblock arXiv:1609.06707v2.

\bibitem[FPRWc]{FormanPalRizzoloWinkel2018b}
Noah Forman, Soumik Pal, Douglas Rizzolo, and Matthias Winkel.
\newblock Interval partition evolutions with emigration related to the {A}ldous
  diffusion.
\newblock arXiv:1804.01205v1.

\bibitem[FPRWd]{FormanPalRizzoloWinkel2018a}
Noah Forman, Soumik Pal, Douglas Rizzolo, and Matthias Winkel.
\newblock Projections of the {A}ldous chain on binary trees: Intertwining and
  consistency.
\newblock arXiv:1802.00862v1.

\bibitem[GPW09]{GrevenPfaffelhuberWinter2009}
Andreas Greven, Peter Pfaffelhuber, and Anita Winter.
\newblock Convergence in distribution of random metric measure spaces
  ({$\Lambda$}-coalescent measure trees).
\newblock {\em Probab.\ Theory Related Fields}, 145:285--322, 2009.

\bibitem[GPW13]{GrevenPfaffelhuberWinter2013}
Andreas Greven, Peter Pfaffelhuber, and Anita Winter.
\newblock Tree-valued resampling dynamics: {M}artingale problems and
  applications.
\newblock {\em Probab.\ Theory Related Fields}, 155(3--4):789--838, 2013.

\bibitem[Kal02]{Kallenberg2002}
Olav Kallenberg.
\newblock {\em Foundations of Modern Probability}.
\newblock Springer, 2002.

\bibitem[KL15]{KliemLoehr2015}
Sandra Kliem and Wolfgang L{\"o}hr.
\newblock Existence of mark functions in marked metric measure spaces.
\newblock {\em Electron.\ J.\ Probab.}, 20(73):1--24, 2015.

\bibitem[L{\"o}h13]{Loehr2013}
Wolfgang L{\"o}hr.
\newblock Equivalence of {Gromov-Prohorov-} and {Gromov's}
  {$\underline\Box_\lambda$}-metric on the space of metric measure spaces.
\newblock {\em Electron.\ Commun.\ Probab.}, 18(17):1--10, 2013.

\bibitem[LVW15]{LoehrVoisinWinter2015}
Wolfgang L{\"o}hr, Guillaume Voisin, and Anita Winter.
\newblock Convergence of bi-measure {$\mathbb R$-trees} and the pruning
  process.
\newblock {\em Ann.\ Inst.\ H.\ Poincar\'e Probab.\ Statist.},
  51(4):1342--1368, 2015.

\bibitem[LW]{LoehrWinter}
Wolfgang L\"ohr and Anita Winter.
\newblock Spaces of algebraic measure trees and triangulations of the circle.
\newblock arXiv:1811.11734v2.

\bibitem[Nus]{Nussbaumer}
Josu{\'e} Nussbaumer.
\newblock The {K}ingman coalescent: old and new aspects.
\newblock
  \url{http://www-stud.uni-due.de/~snjonuss/master_thesis_kingman_coalescent.pdf}.
\newblock Master thesis.

\bibitem[Pal]{Pal2011}
Soumik Pal.
\newblock On the {A}ldous diffusion on continuum tree, {I}.
\newblock arXiv:1104.4186v1.

\bibitem[Sch02]{MR1871950}
Jason Schweinsberg.
\newblock An {$O(n\sp 2)$} bound for the relaxation time of a {M}arkov chain on
  cladograms.
\newblock {\em Random Structures Algorithms}, 20(1):59--70, 2002.

\end{thebibliography}

\end{document}